\documentclass{article}

\usepackage{amsmath}
\usepackage{amssymb}
\usepackage{amsthm}
\usepackage[backref]{hyperref}
\usepackage[linesnumbered,lined,ruled, noend]{algorithm2e}

\usepackage{algpseudocode}
\usepackage{enumerate}
\usepackage[numbers, sort]{natbib}

\usepackage{tikz}
\usepackage{pgfplots}
\pgfplotsset{compat=1.16}
\usetikzlibrary{arrows.meta}
\usetikzlibrary{backgrounds}
\usetikzlibrary{patterns}
\usepgfplotslibrary{patchplots}
\usepgfplotslibrary{fillbetween}
\pgfplotsset{%
    layers/standard/.define layer set={%
        background,axis background,axis grid,axis ticks,axis lines,axis tick labels,pre main,main,axis descriptions,axis foreground%
    }{
        grid style={/pgfplots/on layer=axis grid},%
        tick style={/pgfplots/on layer=axis ticks},%
        axis line style={/pgfplots/on layer=axis lines},%
        label style={/pgfplots/on layer=axis descriptions},%
        legend style={/pgfplots/on layer=axis descriptions},%
        title style={/pgfplots/on layer=axis descriptions},%
        colorbar style={/pgfplots/on layer=axis descriptions},%
        ticklabel style={/pgfplots/on layer=axis tick labels},%
        axis background@ style={/pgfplots/on layer=axis background},%
        3d box foreground style={/pgfplots/on layer=axis foreground},%
    },
}
\usepackage{adjustbox}
\usepackage{subcaption}

\theoremstyle{plain}\newtheorem{theorem}{Theorem}[section]
\theoremstyle{plain}\newtheorem{definition}[theorem]{Definition}
\theoremstyle{plain}\newtheorem{lemma}[theorem]{Lemma}
\theoremstyle{definition}\newtheorem{remark}[theorem]{Remark}
\theoremstyle{definition}\newtheorem{example}[theorem]{Example}
\theoremstyle{plain}\newtheorem{proposition}[theorem]{Proposition}
\theoremstyle{plain}
\theoremstyle{plain}\newtheorem{corollary}[theorem]{Corollary}
\theoremstyle{plain}\newtheorem{assumption}[theorem]{Assumption}

\DeclareMathOperator{\dom}{{dom\,}}

\DeclareMathOperator*{\argmin}{argmin}

\renewcommand{\equiv}{:=}

\newcommand{\set}[2]{\left\{#1\,\left|\,#2\right.\right\}}
\newcommand{\extre}{\overline{\mathbb{R}}}
\newcommand{\xbar}{\bar{x}}
\newcommand{\bbR}{\mathbb{R}}
\newcommand{\bbN}{\mathbb{N}}
\newcommand{\Dcal}{{\mathcal D}}

\newcommand{\Nbb}{\mathbb{N}}
\newcommand{\Mtilde}{{\widetilde{M}}}
\newcommand{\Mbar}{{\overline{M}}}

\begin{document}

\title{First-Order Analysis of Optimization in Uniformly Convex Metric Spaces: Directional Subderivatives and Basic Descent}
\author{D. Russell Luke
\thanks{Institute for Numerical and Applied Mathematics,
    Georg-August University of Göttingen, 37083 Göttingen, Germany. \texttt{r.luke@math.uni-goettingen.de} DRL was supported in part by the Deutsche Forschungsgemeinschaft (DFG, German Research Foundation) - Projektnummer 566257456 and 541767835.},
    Titus Pinta \thanks{ENSTA, 91120 Palaiseau, France. \texttt{titus.pinta@ensta.fr} TP was supported in part by the Fondation Mathématique Jacques Hadamard - Project-ID ANR-22-EXES-0013},~ and
    Qinyu Yan \thanks{Institute for Numerical and Applied Mathematics, Georg-August University of Göttingen. \texttt{qinyu.yan@math.uni-goettingen.de} QY was supported by the Deutsche Forschungsgemeinschaft (DFG, German Research Foundation) - Projektnummer 566257456.}
}
\date{\today}

\maketitle

\begin{abstract}
    We develop tools for the analysis and implementation of explicit first-order methods for minimizing functions in uniformly convex metric spaces.  We formulate sufficient conditions for convergence of descent sequences in terms of directional subderivatives, function values and iterates under regularity assumptions including boundedness, geodesic smoothness and a metric Polyak-Łojasiewicz property. We show that there exists a steepest descent direction in which the assumptions for convergence are satisfied. This work provides a foundation for first-order {\em explicit} algorithms for locally smooth, nonconvex optimization.
\end{abstract}


{\small \noindent {\bfseries 2020 Mathematics Subject Classification:}
    Primary
    49K27,       
    51K05,      
    90C48       
    

    Secondary 
    53C05,     
    51F99,      
    49M27       
}

\bigskip

\section{Introduction}
The setting for our development is a uniquely geodesic, uniformly convex metric space, denoted $(G,d)$. Our central goal is to solve the following optimization problem
\begin{equation}\label{eq: model prb}
    \mathop{\mathrm{minimize}}\limits_{x\in G}\,f(x).
\end{equation}
An algorithmic model generates sequences $(x^k)_{k\in\bbN}$ by taking {\em descent steps} with respect to the function $f$ from the preceding iterate. This is made precise later, but for the moment it is helpful to introduce the notation
\begin{equation}\label{e:descent set vague}
    \Dcal_k:=\bigl\{ y\in G\setminus\{x^k\} \,\big|\, y\,\,\text{is a descent step} \bigr\}.
\end{equation}
The conceptual method is given in Algorithm \ref{alg: BD}.
\begin{algorithm}[ht]
    \caption{Basic descent (BD) method}\label{alg: BD}
    \SetKwInOut{Input}{Parameters}\SetKwInOut{Output}{Initialization}
    \Output{Set $k=0$ and select $x^{0}\in G$.}
    \While{$\Dcal_k\neq \emptyset$}{
        Choose a descent step $y^k\in\Dcal_k$ and step scaling $\lambda_k\in (0,1]$.\\
        Update $x^{k+1}=(1-\lambda_k)x^k\oplus\lambda_ky^k$ and $k=k+1$.
    }
\end{algorithm}
For geodesic metric spaces, in the absence of linear structure {\em implicit tools} using the {\em proximal mapping} \cite{Jost1995} are well-developed and provide notions of descent based on {\em gradient flows} \cite{DeGMarTos80, AmbGigSav2005}. This is applied to evolution systems (see for instance \cite{GigNob21}) and metric measure spaces (see for instance \cite{OhtaPal15}). Purely implicit methods based on proximal splitting are possible whenever the proximal mapping can be computed explicitly \cite{Bacak14, LauLuk21, BLL, LukMir26}. In order to be practical, implicit methods require an ``easy'' representation and implementation of the proximal mappings which, ultimately, are solution mappings to optimization problems. This is limiting, all the more restrictive since almost all of the known ``prox friendly'' functions are convex. Indeed, convexity plays a strong role in approaches based on gradient flows.

Gradient flows are built on two concepts: firstly, metric derivatives of absolutely continuous curves with values in the metric space, and secondly upper gradients of a functional defined in the metric space. Recent approaches pointing to an alternative to gradient flows have been studied on cubical complexes, Hadamard spaces and spaces with curvature bounded above \cite{GooLewLopNic26, GooLewLopNic_cat0cc, GooLewLopNicFoCM, LewLopNic_SIOPT24}. Our main contribution is an approach to notions of descent that are built on the same two concepts, but with a construction that is more appropriate for optimization in uniquely geodesic metric spaces and, moreover, is independent of convexity or notions of a subdifferential. This is demonstrated in Theorem \ref{th: main summary 2}, which provides a {\em first-order} theory for descent methods for smooth optimization. Our development of existence of steepest descent {\em geodesics}, Theorem \ref{th:suff descent - suff}, is distinct from an analogous development of existence of steepest descent curves in metric spaces \cite{DeGMarTos80}.

\subsection{Uniformly convex metric spaces}
We limit our discussion to a uniquely geodesic metric (length) space $(G,d)$, i.e. a metric space where for any $x_0,x_1\in G$ there exists a unique absolutely continuous path of minimal length satisfying $d(\gamma(t),\gamma(s))=|t-s|d(x,y)$ for all $t,s\in[0,1]$. We denote by $x_t:=\gamma(t)=(1-t)x_0\oplus tx_1$ for $t\in[0,1]$ the intermediate points on the geodesic $\gamma$. We may write $\gamma_{[x_0,x_1]}$ to indicate the starting point as well as the end point of the geodesic. We refer to \cite[Section 2]{BurBurIva} for more background.

Throughout this development $(G,d)$ will be a $p$-uniformly convex metric space with $p>1$ and $c>0$ \cite[Definition 2.1]{RuiLopNic15}, 
that is, it is uniquely geodesic and for all $x_0,x_1,y\in G$ and $t\in[0,1]$
\begin{equation}\label{eq: (p,c)-space}
    d(x_t,y)^p\leq (1-t)d(x_0,y)^p+td(x_1,y)^p-\frac{c}{2}t(1-t)d(x_0,x_1)^p,
\end{equation}
where $x_t:=(1-t)x_0\oplus tx_1$ (see also \cite[pp468]{BalCarLie94}). 
The constant $c$ plays no role in the abstract tools, and can be any positive real number, though the only cases that matter in other contexts are $c\in (0,2]$. 
Examples of $p$-uniformly convex spaces are CAT$(\kappa)$ spaces (locally, $p=2$, $c\in(0,2]$), Hadamard spaces ($p=c=2$),
$L^p$ spaces, and uniformly convex Banach spaces.  

The structure of $p$-uniformly convex spaces is used only in Proposition \ref{pr: SD method} to guarantee a sufficient 
decrease condition \eqref{eq: BD assm suff descent} for Algorithm \ref{alg: BD}.  
Sufficient decrease is guaranteed 
in directions of steepest descent which builds on notions of directional 
subderivatives and geodesic slopes;  all of these building blocks are well-defined in 
any uniquely geodesic metric space and do not require the structure of uniform convexity.  

Given $x\in G$, we denote by
\begin{subequations}\label{eq:(G_x,d_x)}
    \begin{equation}\label{eq:G_x}
        G_x:=\bigl\{ \gamma_{[x,y]} \,\big|\, y\in G \bigr\}
    \end{equation}
    the sets of all geodesics starting at $x$. The trivial geodesic $\gamma_{[x,x]}\in G_x$ is defined by $\gamma_{[x,x]}(t)=x$ for all $t\in[0,1]$. It is easily verified that the following mapping defines a metric on $G_x$:
    \begin{equation}\label{eq:d_x}
        d_x:G_x\times G_x\to\mathbb{R} \qquad (\gamma_1,\gamma_2)\mapsto d(\gamma_1(1),\gamma_2(1)).
    \end{equation}
\end{subequations}
Thus $(G_x,d_x)$ is a metric space. We will denote by $G_x^{\circ}:=G_x\setminus\{\gamma_{[x,x]}\}$ the set of all nontrivial geodesics starting at $x\in G$.

The space $(G_x,d_x)$ is not to be confused with a tangent cone at $x$ as defined in \cite[Definition 2.4]{GigNob21}. It is not difficult to show that $(G_x,d_x)$ has the same geometry as $(G,d)$. Indeed, $G_x$ can be viewed as a copy of $G$ but centered at $x$, so we can also call $G_x$ the centered space or the space $G$ centered at $x$. When $(G,d)$ is a $p$-uniformly convex metric space with constants $p$ and $c$, the centered metric space $(G_x,d_x)$ is also $p$-uniformly convex with the same constants, and the geodesic curve between $\gamma_{[x,y_0]},\gamma_{[x,y_1]}$ is given by
\begin{equation}\label{eq: geod btw geods}
    \gamma_t:=\gamma_{[x,y_t]} \quad\text{with}\quad y_t:=(1-t)y_0\oplus ty_1 \qquad\text{for all}\,\,t\in[0,1].
\end{equation}
The closed ball in $G$ centered at $x$ with radius $r>0$ is denoted $B_r(x):=\bigl\{ y\in G \,\big|\, d(x,y)\leq r \bigr\}$. The {\em geodesic interior} of a subset $V\subset G$ is the set
\begin{equation}\label{e:gi}
    \mathrm{gi}\,(V):=\bigl\{ x\in V \,\big|\, \forall\,y\in V\setminus\{x\},\,\exists\,\eta\in G_y^\circ\,\,\text{s.t.}\,\,x=\eta(t)\,\,\text{for}\,\,t\in (0,1) \bigr\}.
\end{equation}
The set of real numbers extended by infinity is denoted  $\extre:=\mathbb{R}\cup\{+\infty\}$, and for a function $f:G\to\extre$, the {\em effective domain} is the set of points where $f$ is finite: $\mathrm{dom}\,(f):=\bigl\{ x\in G \,\big|\, f(x)<+\infty \bigr\}$. For $\alpha\in\mathbb{R}$, the (lower) level-set of $f$ with level $\alpha$ is denoted $\mathrm{lev}_{\leq\alpha}(f):=\set{x\in G}{f(x)\leq\alpha}$, and a function $f:G\to\extre$ is {\em level-bounded} when $\mathrm{lev}_{\leq\alpha}(f)$ is bounded in $G$ for all $\alpha\in\mathbb{R}$. The function $f:G\to\extre$ is proper whenever there is some $\alpha$ for which the lower-level set is not empty. When $\inf_Gf$ is finite and attained, this is denoted $\min_{x\in G}\,f(x)$ and the set of points at which the minimum is attained is denoted $\mathrm{argmin}_G(f)$.

A function $f:G\to\extre$ is called {\em geodesically convex}, or just convex, whenever
$$
f(x_t)\leq (1-t)f(x_0)+tf(x_1) \qquad\text{for all}\,\,x_0,x_1\in G\,\,\text{and}\,\,t\in[0,1].
$$
The function $f$ is called {\em strongly geodesically convex}, or just strongly convex, with constant $\mu>0$, if for all $x_0,x_1\in G$ and all $t\in[0,1]$
\begin{equation}\label{eq: mu-cvx}
    f(x_t)\leq (1-t)f(x_0)+tf(x_1)-\frac{\mu}{2}t(1-t)d(x_0,x_1)^p.
\end{equation}
\subsection{Overview and main results}
With the notation and basic objects above, we can state the main results. Some objects in the statements, like subderivatives, are defined later.
\begin{theorem}[convergence of directional subderivatives]\label{th: main summary 1}
    Let $(G,d)$ be $p$-uniformly convex and $f:G\to\mathbb{R}$. Let $(x^k)_{k\in\bbN}$ be a sequence generated by Algorithm \ref{alg: BD} and let $(y^k,\lambda_k)_{k\in\bbN}$ be the corresponding sequence of descent steps and step scalings. Assume the following:
    \begin{enumerate}
        \item[(a)] $f$ has at least one minimizer on $U$, i.e. $\mathrm{argmin}_G(f)\neq\emptyset$;
        \item[(b)] $f$ has local uniform Lipschitz gradient along geodesics with constant $L>0$ on the convex subset $U$ with $\mathrm{argmin}_G(f)\subset\mathrm{gi}(U)$;
        \item[(c)] the descent steps and step scalings $(y^k, \lambda_k)\in U\times \mathbb{R}_+$ satisfy
        $$
        f((1-\lambda_k)x^k\oplus\lambda_k y^k)<f(x^k) \qquad \text{for all}\,\,k\in\bbN,
        $$
        and whenever Algorithm \ref{alg: BD} does not terminate
        $$
        \sum_{k=0}^{+\infty}\lambda^2_kd(x^k, y^k)^2<+\infty \qquad\text{while}\qquad \sum_{k=0}^{+\infty}\lambda_kd(x^k, y^k)=+\infty.
        $$
    \end{enumerate}
    Then at least one of the following two events must occur:
    \begin{enumerate}
        \item[(i)]\label{th: main summary 1i} the sequence $(x^k)_{k\in\bbN}$ converges finitely to a point $\xbar\in U$ where the directional subderivative (Definition \ref{de: geod slp/dir devt/norm dir devt}) is nonnegative:
        $$
        d^-f(\bar{x})[\gamma_{[\bar{x}, y]}]\geq 0 \qquad\text{for all}\,\,y\in U;
        $$
        \item[(ii)]\label{th: main summary 1ii} the sequence of directional subderivatives converges to zero:
        $$
        \lim_{k\to\infty} d^-f(x^k)[\gamma_{[x^k, y^k]}]=0.
        $$
    \end{enumerate}
    If the sequence terminates finitely at $\xbar$, that is case (i) occurs, and $f$ is convex,
    then $\xbar\in\mathrm{argmin}_G(f)$, i.e. $\bar{x}$ is a global minimizer of $f$.
\end{theorem}
\begin{theorem}[convergence of first-order methods for smooth optimization]\label{th: main summary 2}
    Let $(G,d)$ be $p$-uniformly convex and $f:G\to\mathbb{R}$. Let $(x^k)_{k\in\bbN}$ be a sequence generated by Algorithm \ref{alg: BD}. Assume the following:
    \begin{enumerate}
        \item[(a)] the sequence $(x^k)_{k\in\bbN}\subset U$, a convex subset of $G$, and $f$ possesses at least one local minimum, $x^*$, on $U$;
        \item[(b)] the sequence $(x^k)_{k\in\bbN}$ satisfies the {\em sufficient decrease condition} relative to the steepness $|D|f$ (Definition \ref{d:steepness}): there exists $\delta>0$ such that
        \begin{equation}\label{eq: BD assm suff descent}
            f(x^{k+1})-f(x^k)\leq-\frac{1}{\delta}|D|f(x^k)^{\frac{2}{p-1}} \qquad\text{for all}\,\,k\in\bbN;
        \end{equation}
        \item[(c)] $f$ satisfies the Polyak-Łojasiewicz property \eqref{eq: PL} at the local minimum $x^*$ with constant $C\in(0,\delta)$ for the neighborhood $U$.
    \end{enumerate}
    Then the sequence of function values $(f(x^k))_{k\in\bbN}$ converges at least linearly to the local minimum $f_*=f(x^*)$ with
    $$
    f(x^k)-f_*\leq\Bigl(f(x^0)-f_*\Bigr)\left(1-\tfrac{C}{\delta}\right)^k \qquad\text{for all}\,\,k\in\bbN.
    $$
    In addition to the assumptions above, let the metric space $(G,d)$ be complete, and suppose that
    \begin{enumerate}
        \item[(d)] $f$ has local uniform Lipschitz gradient along geodesics with constant $L>0$ on $U$, and
        \item[(e)] the sequence $(x^k)_{k\in\bbN}$ is stable, i.e.
        \begin{equation}\label{e:suff iterate descent}
            d(x^{k+1},x^k)\leq\frac{1}{\sqrt{\delta}}|D|f(x^k)^{\frac{1}{p-1}} \qquad\text{for all}\,\,k\in\bbN.
        \end{equation}
    \end{enumerate}
    Then the sequence of iterates $(x^k)$ converges at least linearly to a local minimizer $\xbar\in\mathrm{argmin}_U(f)$ with
    $$
    d(x^k,\xbar)\leq\frac{\sqrt{f(x^0)-f_*}}{1-\sqrt{1-\tfrac{C}{\delta}}}\left(\sqrt{1-\tfrac{C}{\delta}}\right)^{k} \qquad\text{for all}\,\,k.
    $$
\end{theorem}
\begin{theorem}[steepest descent directions and sufficient decrease]\label{th:suff descent - suff}
    Let $(G,d)$ be boundedly compact, $p$-uniformly convex with $p\in(1,2]$. Let $(x^k)_{k\in\bbN}$ be a sequence generated by Algorithm \ref{alg: BD} and let $(y^k,\lambda_k)_{k\in\bbN}$ be the corresponding sequence of descent steps and step scalings. Assume that
    \begin{enumerate}[(i)]
        \item $f:G\to\mathbb{R}$ is level-bounded and has local uniform Lipschitz gradient along geodesics with constant $L>0$ on the convex subset $U$;
        \item $\mathrm{lev}_{\leq f(x^0)}(f)\subset \mathrm{gi}\,(U)$ with distance at least $\rho>0$ to the boundary of $U$: i.e. $\bigcup_{x\in \mathrm{lev}_{\leq f(x^0)}(f)} B_\rho(x)\subset \mathrm{gi}\,(U)$;
        \item in each step $k$ with $\mathcal{D}_k\neq\emptyset$ the descent step $y^k\in\mathcal{D}_k$ is a {\em steepest descent step} with $d(x^k,y^k)=\rho$, i.e.
        $$
        Df(x^k)[\gamma_{[x^k,y^k]}]=-|D|f(x^k) \quad\text{and}\quad d(x^k,y^k)=\rho.
        $$
    \end{enumerate}
    Then there exists a uniform constant $\underline{\delta}>0$ such that for all $\delta\geq\underline{\delta}$ and all $k$, $[\alpha_k,\beta_k]\cap(0,1]\neq\emptyset$ where $\alpha_k=\alpha(x^k, y^k)$ and $\beta_k=\beta(x^k, y^k)$ are given in \eqref{eq: sd condition}, and whenever
    \begin{equation}\label{eq: admissable lambda}
        \lambda_k\in [\alpha_k,\beta_k]\cap(0,1],
    \end{equation}
    the {\em sufficient decrease condition} \eqref{eq: BD assm suff descent} holds.
\end{theorem}
\section{Directional derivatives and optimality}\label{sec: dir dervt}
\subsection{Definitions and basic properties}\label{s:defs}
Our construction of the directional subderivative  follows the same logic of the construction in Euclidean spaces as a limit over positive step scalings and descent steps (see for instance \cite[Definition 8.1]{VA}). This is also used in \cite[Definition 3.2]{OhtaPal15}, but since we make extensive use of both pieces of this limiting procedure, we call the first limit with respect to positive step scalings the {\em geodesic slope} along the given geodesic.
\begin{definition}[geodesic slope, directional derivative and normalized directional derivatives]\label{de: geod slp/dir devt/norm dir devt}
    Let $f:G\to\overline{\mathbb{R}}$ and $x\in\mathrm{dom}\,(f)$, and fix any $\gamma\in G_x$.
    \begin{enumerate}[(i)]
        \item The {\em geodesic sub- and superslope} of $f$ at $x$ along $\gamma$ are
        $$
        S^-f(x)[\gamma]:=\liminf_{t\searrow 0}\frac{f(\gamma(t))-f(x)}{t}, \, S^+f(x)[\gamma]:=\limsup_{t\searrow 0}\frac{f(\gamma(t))-f(x)}{t}.
        $$
        \item The {\em directional sub- and superderivative} of $f$ at $x$ in direction $\gamma$ are
        $$
        d^-f(x)[\gamma]:=\liminf_{\eta\to\gamma}S^-f(x)[\eta], \quad d^+f(x)[\gamma]:=\limsup_{\eta\to\gamma}S^+f(x)[\eta],
        $$
        and $f$ is said to be {\em directionally sub- resp. superdifferentiable at $x$ in direction $\gamma$} whenever these limits are finite and coincide:
        $$
        d^-f(x)[\gamma]=S^-f(x)[\gamma]\neq\pm\infty, \quad d^+f(x)[\gamma]=S^+f(x)[\gamma]\neq\pm\infty.
        $$
        \item The {\em normalized directional sub- and superderivative} of $f$ at $x$ in direction $\gamma\in G_x^{\circ}$ are
        $$
        D^-f(x)[\gamma]:=\frac{d^-f(x)[\gamma]}{d(\gamma(0), \gamma(1))}, \quad  D^+f(x)[\gamma]:=\frac{d^+f(x)[\gamma]}{d(\gamma(0), \gamma(1))}.
        $$
        The function $f$ is said to be
        {\em directionally differentiable at $x$ in direction $\gamma$} whenever the directional sub- and superderivatives coincide:
        $$
        d^+f(x)[\gamma]=d^-f(x)[\gamma]=df(x)[\gamma]:=\lim_{t\searrow 0}\frac{f(\gamma(t))-f(x)}{t}\neq\pm\infty.
        $$
        The limit $df(x)[\gamma]$ is the {\em directional derivative of $f$ at $x$ in direction $\gamma$}.
        The {\em normalized directional derivative} is denoted $Df(x)[\gamma]$.
    \end{enumerate}
\end{definition}
These objects are not limited to descent, so unlike the slope defined in \cite[Definition 2.1.4]{AmbGigSav2005} the sign of the directional derivatives is relevant. The following chain of inequalities is elementary:
\begin{equation}\label{eq: eleme ineqs}
    d^-f(x)[\gamma]\leq S^-f(x)[\gamma]\leq S^+f(x)[\gamma]\leq d^+f(x)[\gamma]\quad\forall x\in G, \gamma\in G_x.
\end{equation}
The next proposition will play an important role in the analysis of optimality conditions.
\begin{proposition}[directional subderivative of (strongly) convex functions]\label{pr: dir dervt of cvx fct}
    Let $f:G\to\overline{\mathbb{R}}$ be convex  and $x\in\mathrm{dom}\,(f)$.  Then
    $$
    f(y)-f(x)\geq d^-f(x)[\gamma_{[x,y]}] \qquad\text{for all}\,\,y\in G.
    $$
    If $f$ is strongly convex with constant $\mu>0$, then
    $$
    f(y)-f(x)\geq d^-f(x)[\gamma_{[x,y]}]+\frac{\mu}{2}d(x,y)^p \qquad\text{for all}\,\,y\in G.
    $$
\end{proposition}
\begin{proof}
    Let $\gamma_{[x,y]}\in G_x$ and denote $x_t:=\gamma_{[x,y]}(t)$. By definition of convexity, for some $\mu\geq 0$
    $$
    f(x_t)\leq (1-t)f(x)+tf(y)-\frac{\mu}{2}(1-t)td(x,y)^p \qquad\text{for all}\,\,t\in(0,1).
    $$
    Therefore, for all $t\in(0,1)$ and some $\mu\geq 0$
    $$
    \frac{f(x_t)-f(x)}{t}\leq f(y)-f(x)-\frac{\mu}{2}(1-t)d(x,y)^p,
    $$
    hence
    \begin{align*}
        f(y)-f(x)-\frac{\mu}{2}d(x,y)^p&=\liminf_{t\searrow 0}\left(f(y)-f(x)-\frac{\mu}{2}(1-t)d(x,y)^p\right)\\
        &\geq\liminf_{t\searrow 0}\frac{f(x_t)-f(x)}{t}=S^-f(x)[\gamma_{[x,y]}]\\
        &\geq\liminf_{\eta\to\gamma_{[x,y]}}S^-f(x)[\eta]=d^-f(x)[\gamma_{[x,y]}].
    \end{align*}
\end{proof}
As in the Euclidean setting, the directional derivatives enjoy sum and scalar multiplication rules.
\begin{proposition}[sum and product rules]\label{pr: sum rules}
    Given functions $f,g:G\to\extre$, fix any $x\in\mathrm{dom}\,(f)\cap\mathrm{dom}\,(g)$ and $\gamma\in G_x$. Then
    \begin{equation}\label{eq: sum rules 1}
        \begin{aligned}
            &d^-(f+g)(x)[\gamma]\geq  d^-f(x)[\gamma]+d^-g(x)[\gamma]\\
            &d^+(f+g)(x)[\gamma]\leq  d^+f(x)[\gamma]+d^+g(x)[\gamma]
        \end{aligned}
    \end{equation}
    and
    \begin{equation}\label{eq: multi rules 1}
        d^-(\lambda f)(x)[\gamma]=\lambda d^-f(x)[\gamma], \quad d^+(\lambda f)(x)[\gamma]=\lambda d^+f(x)[\gamma].
    \end{equation}
    Furthermore, if $f,g$ are both directionally differentiable at $x$ in direction $\gamma$, then for any $\lambda\in \mathbb{R}$ the composite function $\lambda f+g$ is also directionally differentiable at $x$ in direction $\gamma$ with
    \begin{equation}\label{eq: sum rules 2}
        d(\lambda f+g)(x)[\gamma]=\lambda df(x)[\gamma]+dg(x)[\gamma].
    \end{equation}
\end{proposition}
\begin{proof}
    Inequalities in \eqref{eq: sum rules 1} follow from the super- and subadditivity of $\liminf$ and $\limsup$, respectively. Indeed,
    \begin{align*}
        d^-(f+g)(x)[\gamma]&=\liminf_{\eta\to\gamma}S^-(f+g)(x)[\gamma]\\
        &=\liminf_{\eta\to\gamma}\liminf_{t\searrow 0}\frac{(f+g)(\eta(t))-(f+g)(x)}{t}\\
        &\geq\liminf_{\eta\to\gamma}\Bigl(S^-f(x)[\eta]+S^-g(x)[\eta]\Bigr)\geq d^-f(x)[\gamma]+d^-g(x)[\gamma].
    \end{align*}
    The proof for the superderivative is similar. The proof of the multiplication rule \eqref{eq: multi rules 1} follows immediately from the definitions. Indeed,
    \begin{align*}
        d^-(\lambda f)(x)[\gamma]&=\liminf_{\eta\to\gamma}S^-(\lambda f)(x)[\gamma]=\liminf_{\eta\to\gamma}\liminf_{t\searrow 0}\frac{(\lambda f)(\gamma(t))-(\lambda f)(x)}{t}\\
        &=\lambda\liminf_{\eta\to\gamma}\liminf_{t\searrow 0}\frac{f(\gamma(t))-f(x)}{t}\\
        &=\lambda\liminf_{\eta\to\gamma}S^-f(x)[\gamma]=\lambda d^-f(x)[\gamma].
    \end{align*}
    Note that we can move $\lambda$ outside of $\liminf$ because $\lambda\geq 0$. The proof for the superderivative is similar.

    If the $\liminf$ is replaced by $\lim$, then the linearity of $\lim$ yields the multiplication rule
    \begin{align*}
        d(\lambda f)(x)[\gamma]&=\lim_{t\searrow 0}\frac{(\lambda f)(\gamma(t))-(\lambda f)(x)}{t}=\lim_{t\searrow 0}\frac{\lambda\bigl(f(\gamma(t))-f(x)\bigr)}{t}\\
        &=\lambda\lim_{t\searrow 0}\frac{f(\gamma(t))-f(x)}{t}=\lambda df(x)[\gamma].
    \end{align*}
    If both $f$ and $g$ are directionally differentiable at $x$ in direction $\gamma$, then \eqref{eq: sum rules 1} yields
    \begin{equation}\label{eq: sum rules 3}
        \begin{aligned}
            d^-(\lambda f+g)(x)[\gamma]&\geq \lambda d^-f(x)[\gamma]+d^-g(x)[\gamma]=\lambda df(x)[\gamma]+dg(x)[\gamma]\\
            &=\lambda d^+f(x)[\gamma]+d^+g(x)[\gamma]\geq d^+(\lambda f+g)(x)[\gamma].
        \end{aligned}
    \end{equation}
    At the same time, $d^-(\lambda f+g)(x)[\gamma]\leq d^+(\lambda f+g)(x)[\gamma]$. Therefore, \eqref{eq: sum rules 3} and \eqref{eq: eleme ineqs} applied to $\lambda f+g$ hold with equality everywhere, yielding
    \begin{align*}
        d^+(\lambda f+g)(x)[\gamma]&=S^+(\lambda f+g)(x)[\gamma]=\lambda df(x)[\gamma]+dg(x)[\gamma]\\
        &=S^-(\lambda f+g)(x)[\gamma]=d^-(\lambda f+g)(x)[\gamma].
    \end{align*}
    Hence, the limit $d(\lambda f+g)(x)[\gamma]$ exists and the equality \eqref{eq: sum rules 2} is proved.
\end{proof}
The same argument as in the proof of Proposition \ref{pr: sum rules} also proves the sum rules for geodesic slopes and normalized directional derivatives.
\begin{remark}\label{rm: truncation of geod}
    Given two points $x,y\in G$ and fix $z:=\gamma_{[x,y]}(\lambda)$ for some $\lambda\in(0,1)$
    $$
    \gamma_{[x,z]}(t)=\gamma_{[x,y]}(\lambda t) \qquad\text{and}\qquad \gamma_{[z,y]}(t)=\gamma_{[x,y]}(\lambda+(1-\lambda)t)
    $$
    by the property of geodesic path. One can easily verify that they are constant speed geodesics. Moreover, if $x\in\mathrm{dom}\,(f)$ then
    $$
    S^-f(x)[\gamma_{[x,z]}]=\lambda S^-f(x)[\gamma_{[x,y]}].
    $$
    If we normalize the slopes by the length of geodesics, then it is immediate that $S^-f(x)[\gamma_{[x,y]}]/d(x,y)=S^-f(x)[\gamma_{[x,z]}]/d(x,z)$. Next, observe that if $\eta^k\to\gamma_{[x,y]}$, then $\eta^k(\lambda)\to z$ by the uniqueness of geodesics. Therefore,
    $$
    df(x)[\gamma_{[x,z]}]=\liminf_{\eta\to\gamma_{[x,z]}}S^-f(x)[\eta]\leq\liminf_{\eta\to\gamma_{[x,y]}}S^-f(x)[\eta_{[x,\eta(\lambda)]}]=\lambda df(x)[\gamma_{[x,y]}].
    $$
    If in addition $f$ is directionally differentiable at $x$ in the direction $\gamma_{[x,y]}$, then by linearity $df(x)[\gamma_{[x,z]}]=\lambda df(x)[\gamma_{[x,y]}]$ and $Df(x)[\gamma_{[x,y]}]=Df(x)[\gamma_{[x,z]}]$.
\end{remark}
\subsection{First-order optimality conditions}\label{s:foc}
Returning to the unconstrained optimization problem \eqref{eq: model prb}, the next theorem states the first-order necessary and sufficient conditions for minimality.
\begin{theorem}[first-order necessary and sufficient conditions for minimality]\label{th: FNC-P/FSC}
    Let $f:G\to\overline{\mathbb{R}}$ be proper. If $\bar{x}\in\mathrm{dom}\,(f)$ is a local minimizer of $f$, then
    \begin{equation}\label{eq: Fermat with dervt}
        d^-f(\bar{x})[\gamma_{[\bar{x},y]}]\geq 0 \qquad\text{for all}\,\,y\in G.
    \end{equation}
    Conversely, if $f$ is proper and convex, with $\bar{x}\in\mathrm{dom}\,(f)$ satisfying \eqref{eq: Fermat with dervt} then $\bar{x}$ is a global minimizer of $f$. If additionally $f$ is strongly convex, then $\bar{x}$ is the unique minimizer of $f$.
\end{theorem}
\begin{proof}
    Let $r>0$ be sufficiently small such that
    $$
    f(y)\geq f(\bar{x}) \qquad\text{for all}\,\,y\in B_r(\bar{x}).
    $$
    Fix some $y\in B_r(\bar{x})$ and let $x_t:=\gamma_{[\bar{x},y]}(t)$. Since
    $$
    d(\bar{x},x_t)=td(\bar{x},y)\leq tr\leq r\qquad\text{for all}\,\,t\in(0,1].
    $$
    the local minimality of $\bar{x}$ shows
    $$
    \frac{f(x_t)-f(\bar{x})}{t}\geq 0 \qquad\text{for all}\,\,t\in(0,1].
    $$
    This implies
    \begin{equation}\label{eq: FNC-P 1}
        S^-f(\bar{x})[\gamma_{[\bar{x},y]}]=\liminf_{t\searrow 0}\frac{f(x_t)-f(\bar{x})}{t}\geq 0 \qquad\text{for all}\,\,y\in B_r(\bar{x}).
    \end{equation}
    Now fix any $y\in G$. There exists some $t\in(0,1]$ such that $z:=\gamma_{[\bar{x},y]}(t)\in B_r(\bar{x})$. Remark \ref{rm: truncation of geod} and \eqref{eq: FNC-P 1} yield
    $$
    S^-f(\bar{x})[\gamma_{[\bar{x},y]}]=\frac{1}{t}S^-f(\bar{x})[\gamma_{[\bar{x},z]}]\geq 0.
    $$
    Since $y$ was arbitrary, this shows that $S^-f(x)[\gamma]\geq 0$ for all $\gamma\in G_x$,  and hence $d^-f(x)[\gamma]\geq 0$, as claimed.

    For the converse statement, Proposition \ref{pr: dir dervt of cvx fct} yields
    $$
    f(y)-f(\bar{x})\geq d^-f(\bar{x})[\gamma_{[\bar{x},y]}]+\frac{\mu}{2}d(\bar{x},y)^p \qquad\text{for all}\,\,y\in G
    $$
    with $\mu\geq 0$. Since by assumption $d^-f(\bar{x})[\gamma_{[\bar{x},y]}]\geq 0$, we have
    \begin{equation}\label{eq: p-th growth}
        f(y)-f(\bar{x})\geq \frac{\mu}{2}d(\bar{x},y)^p \qquad\text{for all}\,\,y\in G,
    \end{equation}
    i.e.  $\bar{x}\in G$ is a global minimizer of $f$. If additionally $\mu>0$, then $f(\bar{x})<f(y)$ whenever $y\in G\setminus\{\bar{x}\}$.
\end{proof}
Theorem \ref{th: FNC-P/FSC} states that a strongly convex function achieves a pointwise  $p$-th order growth \eqref{eq: p-th growth} at its minimizers. Furthermore, from the proof of  Theorem \ref{th: FNC-P/FSC} it can be seen that the optimality condition \eqref{eq: Fermat with dervt} can be relaxed to
\begin{equation}\label{eq: Fermat with slp}
    S^-f(\bar{x})[\gamma_{[\bar{x},y]}]\geq 0 \qquad\text{for all}\,\,y\in G.
\end{equation}
\section{Basic descent method}\label{sec: BD method}
The descent set $\Dcal_k$ that was vaguely specified in \eqref{e:descent set vague} can now be defined  precisely:
\begin{equation}\label{e:descent set}
    \Dcal_k:=\bigl\{ y\in G\setminus\{x^k\} \,\big|\, D^-f(x^k)[\gamma_{[x^k,y]}]<0 \bigr\}.
\end{equation}
\subsection{Generic convergence analysis}\label{s:gen convergence}
In this subsection, we prove a generic result of the BD method. Note that the geodesic property implies that the step scaling $\lambda_k$ in Algorithm \ref{alg: BD} is given by
$$
d(x^{k+1},x^k)=\lambda_k d(x^k,y^k) \qquad\text{for all}\,\,k\in\Nbb.
$$
The following generic assumptions on the sequence $(x^k)_{k\in\bbN}$ in relation to the function $f$ will be used to classify the kinds of regularity of $f$ that are required to achieve convergence results.
\begin{assumption}[generic descent assumptions]\label{as: gen assm}
    Let $(x^k)_{k\in\bbN}$ be a sequence generated by Algorithm \ref{alg: BD} and let $(y^k,\lambda_k)_{k\in\bbN}$ be the corresponding sequence of descent steps and step scalings.
    \begin{enumerate}[(a)]
        \item\label{as: gen assm bbb} Boundedness: $f$ is bounded below on $G$.
        \item\label{as: gen assm ds} Decreasing steps: whenever Algorithm \ref{alg: BD} does not terminate,
        the step lengths satisfy
        $$
        \sum_{k=0}^{+\infty}\lambda^2_kd(x^k, y^k)^2<+\infty \qquad\text{while}\qquad \sum_{k=0}^{+\infty}\lambda_kd(x^k, y^k)=+\infty.
        $$
        \item\label{as: gen assm mdfv} Monotone decreasing function values:
        $$
        f((1-\lambda_k)x^k\oplus\lambda_k y^k)<f(x^k) \quad \text{for all}\,\,k\in\bbN.
        $$
        \item\label{as: gen assm sd} Sufficient decrease: there exists $M>0$ and $\sigma\in(0,1]$ such that for all $k\in\bbN$
        $$
        f((1-\lambda_k)x^k\oplus\lambda_k y^k)\leq f(x^k)+\lambda_k \sigma d^-f(x^k)[\gamma_{[x^k, y^k]}]+\lambda_k^2 Md(x^k, y^k)^2.
        $$
    \end{enumerate}
\end{assumption}
\begin{lemma}[generic convergence of the BD method]\label{lm: proto descent}
    Let $f:G\to\extre$ be proper. Let $(x^k)_{k\in\bbN}$ be the sequence generated by Algorithm \ref{alg: BD} with initial point $x^0\in\mathrm{dom}\,(f)$. If the conditions in Assumption \ref{as: gen assm} hold, then
    $$
    \lim_{k\to\infty}d(x^{k+1},x^k)=0,
    $$
    and at least one of the following must occur:
    \begin{enumerate}[(i)]
        \item\label{lm: proto descent a} the sequence $(x^k)$ converges finitely to a point $\xbar\in G$ satisfying
        $$
        d^-f(\xbar)[\gamma_{[\xbar, y]}]\geq 0 \qquad\text{for all}\,\,y\in G;
        $$
        \item\label{lm: proto descent b} the sequence of directional subderivatives converges to zero:
        $$
        \lim_{k\to\infty} d^-f(x^k)[\gamma_{[x^k, y^k]}]=0.
        $$
    \end{enumerate}
\end{lemma}
\begin{proof}
    That the sequence of steps $d(x^{k+1},x^k)$ converges to zero follows immediately from either finite termination or Assumption \ref{as: gen assm} \eqref{as: gen assm ds} since $d(x^{k+1},x^k)=\lambda_kd(x^k,y^k)$. Next note that by Assumption \ref{as: gen assm} \eqref{as: gen assm mdfv}, the sequence of function values is monotone decreasing.

    With these initial facts in mind, if the sequence of points $(x^k)_{k\in\bbN}$ converges finitely, then the set $\Dcal_K$ must be empty for some $K$. But $\Dcal_K=\emptyset$ is equivalent to $D^-f(x^K)[\gamma]\geq 0$ for all $\gamma\in G^\circ_{x^K}$, which  for all $\gamma\in G^\circ_{x^K}$ is equivalent to $0\leq d^-f(x^K)[\gamma]/d(\gamma(0), \gamma(1))$. Finally, this is equivalent to $d^-f(x^K)[\gamma_{[x^K,y]}]\geq 0$ for all $y\in G$. By Theorem \ref{th: FNC-P/FSC} the function then satisfies first-order necessary conditions for local optimality and $x^K$ is therefore a critical point. Nevertheless case \eqref{lm: proto descent b} might still occur.

    If case \eqref{lm: proto descent a} does not happen, it remains to show that case \eqref{lm: proto descent b} must occur. By Assumption \ref{as: gen assm} \eqref{as: gen assm bbb}, the function values are bounded below by some $f_*\leq f(x^k)$ for all $k$. Moreover, since the sequence $(x^k)_{k\in\bbN}$ does not converge finitely, there is some $ y^k\in B_r(x^k)$ for all $k\in\bbN$ for which the sequence of corresponding function values is monotonically decreasing.  It follows that $f(x^{k+1})-f(x^k)\to 0$. In this context, note that Assumption \ref{as: gen assm} \eqref{as: gen assm sd} yields
    $$
    f(x^{k+1})-f(x^{k})\leq\lambda_k\sigma d^-f(x^k)[\gamma_{[x^k,y^k]}]+\lambda_k^2d(x^k,y^k)^2M \qquad \text{for all}\,\, k\in\Nbb,
    $$
    where $\sigma\in(0,1]$ is fixed. Summing these differences over $0\leq j\leq k-1$ yields
    \begin{align*}
        f(x^k)-f(x^0)&\leq\sum_{j=0}^{k-1}\lambda_j \sigma d^-f(x^j)[\gamma_{[x^j,y^j]}]+\sum_{j=0}^{k-1}\lambda_j^2d(x^j,y^j)^2M.
    \end{align*}
    Then square summability of the steps $\lambda_jd(x^j,y^j)$ yields
    $$
    f_*-f(x^0)\leq f(x^k)-f(x^0)\leq\sum_{j=0}^{k-1}\lambda_j \sigma d^-f(x^j)[\gamma_{[x^j,y^j]}]+\Mbar,
    $$
    where $\Mbar:=M\sum_{j\in\bbN}\lambda_j^2d(x^j,y^j)^2$, so there is a lower bound $\Mtilde<0$ for the monotonically decreasing partial sum: $\Mtilde\leq\sum_{j=0}^k\lambda_jd^-f(x^j)[\gamma_{[x^j,y^j]}]$ for all $k\in\bbN$. Hence the partial sums converge to $a^*:=\sigma\lim_{k\to\infty}\sum_{j=0}^k\lambda_j d^-f(x^j)[\gamma_{[x^j,y^j]}]$. Since the scalings $\lambda_j$ are nonnegative and the summands  are negative, this implies that
    \begin{align*}
        &\sum_{j=0}^{+\infty}\lambda_j d^-f(x^j)[\gamma_{[x^j,y^j]}]
        =\,\sum_{j=0}^{+\infty}\lambda_j D^-f(x^j)[\gamma_{[x^j,y^j]}]d(x^j,y^j)\neq\pm\infty,
    \end{align*}
    and therefore $D^-f(x^k)[\gamma_{[x^k,y^k]}]\to 0$ by the second part of Assumption~\ref{as: gen assm} \eqref{as: gen assm ds}. Moreover $d^-f(x^k)[\gamma_{[x^k,y^k]}]\to 0$, which
    is exactly the case \eqref{lm: proto descent b}.
\end{proof}
It is not guaranteed that event \eqref{lm: proto descent b} means that the sequence of points $(x^k)_{k\in\bbN}$ converges to critical points: the directional subderivatives in the descent steps $y^k$ could be vanishing too quickly, in which case the iterates simply stagnate.  To prevent this, we need the sufficient decrease condition \eqref{eq: BD assm suff descent}.
\subsection{Uniform Lipschitz gradient along geodesics}\label{s:Unif Lip}
This subsection characterizes a family of functions that satisfy the assumptions, and also establishes criteria to guarantee convergence to critical points. From this point forward the functions are defined as mappings $f:G\to\mathbb{R}$. As such, the functions are finite-valued everywhere on $G$, so trivially $\dom f=G$, and this detail can be ignored.

For $\gamma_{[x,y]}\in G_x^{\circ}$ and $f:G\to\mathbb{R}$, the notation $f\in C^m(\gamma_{[x,y]})$ for $m\in\Nbb$, means that the composed function $f\circ\gamma_{[x,y]}:[0,1]\to\mathbb{R}$, $t\mapsto f(\gamma_{[x,y]}(t))$ is of $C^m$ class, i.e. $f\circ\gamma_{[x,y]}\in C^m([0,1])$.

In the following, we denote by $C(G)$ the set of all finite and continuous functions $f:G\to\mathbb{R}$; the set of continuously differentiable functions with respect to the parameterization of the geodesic $\gamma_{[x,y]}$ is denoted $C^1(\gamma_{[x,y]})$.
\begin{proposition}\label{pr: C^1 implies diff}
    Let $\gamma_{[x,y]}\in G_x^{\circ}$ and $f\in C(G)\cap C^1(\gamma_{[x,y]})$, then $f$ is directionally differentiable at $x\in G$ in the direction $\gamma_{[x,y]}$, and for all $t_0\in(0,1)$
    \begin{equation}\label{eq: df(z)}
        \left.\frac{\mathrm{d}}{\mathrm{d}t}f(\gamma_{[x,y]}(t))\right|_{t=t_0}=\frac{df(\gamma(t_0))[\gamma_{[\gamma(t_0),y]}]}{1-t_0}=Df(\gamma(t_0))[\gamma_{[\gamma(t_0),y]}]d(x,y).
    \end{equation}
    At $t_0=0$, $df(\gamma(t_0))[\gamma_{[x,y]}]$ coincides with the right-derivative of $(f\circ\gamma_{[x,y]})$.
\end{proposition}
\begin{proof}
    Directional differentiability of $f$ in the direction $\gamma_{[x,y]}$ follows immediately from the restriction $f\in C(G)\cap C^1(\gamma_{[x,y]})$.

    To see \eqref{eq: df(z)}, let $z:=\gamma_{[x,y]}(t_0)$ and $g(t)\equiv f(\gamma_{[x,y]}(t))$. By Remark \ref{rm: truncation of geod} $\gamma_{[z,y]}(t)=\gamma_{[x,y]}(t_0+(1-t_0)t)$ for all $t\in[0,1]$. This, together with continuity of $f$ and existence of $g'(t_0)$ yields
    \begin{align*}
        g'(t_0)&=\lim_{\Delta t\to 0}\frac{g(t_0+\Delta t)-g(t_0)}{\Delta t}=\lim_{\Delta t\searrow 0}\frac{f(\gamma_{[x,y]}(t_0+\Delta t))-f(z)}{\Delta t}\\
        &=\lim_{\Delta t\searrow 0}\frac{f(\gamma_{[x,y]}(t_0+(1-t_0)\Delta t))-f(z)}{(1-t_0)\Delta t}\\
        &=\frac{1}{1-t_0}\lim_{\Delta t\searrow 0}\frac{f(\gamma_{[z,y]}(\Delta t))-f(z)}{\Delta t}=\frac{df(z)[\gamma_{[z,y]}]}{1-t_0}.
    \end{align*}
    The last equality follows from the fact that $f$ is also $C^1$ along $\gamma_{[z,y]}$. Moreover, by the property of geodesics, $d(z,y)/d(x,y)=1-t_0$. Thus for all $t_0\in(0,1)$
    $$
    \frac{df(z)[\gamma_{[z,y]}]}{1-t_0}=\frac{df(z)[\gamma_{[z,y]}]d(x,y)}{d(z,y)}=Df(z)[\gamma_{[z,y]}]d(x,y).
    $$
    For $t_0=0$ (i.e. $z=x$), the directional derivative $df(x)[\gamma_{[x,y]}]$ coincides with the right-derivative of $g$ at $t_0=0$ by definition of the one-side derivative.
\end{proof}
The next definition introduces a uniform regularity of $f$ on a convex subset $U\subset G$. We denote
$$
C^1(U):=\bigl\{ f\in C(U) \,\big|\, f\in C^1(\gamma_{[x,y]})\,\,\text{for all distinct}\,\,x,y\in U \bigr\}.
$$
Let $x_t:=\gamma_{[x,y]}(t)$ and $x_s:=\gamma_{[x,y]}(s)$. A function $f:G\to\mathbb{R}$ is said to have {\em local uniform Lipschitz gradient along geodesics on a convex set $U$} if
\begin{equation}\label{eq: L-gradient}
    \begin{aligned}
        f\in C^1(U)\mbox{ and } &\exists L>0\,\,\text{such that}\,\,\forall x,y\in U,\,\forall\,t,s\in[0,1) \\
        &\bigl|Df(x_t)[\gamma_{[x_t,y]}]-Df(x_s)[\gamma_{[x_s,y]}]\bigr|\leq Ld(x_t,x_s).
    \end{aligned}
\end{equation}
Functions having local uniform Lipschitz gradient along geodesics constitutes a 
rich class on manifolds, but for the metric setting 
it is important that the property is only required locally. 
In Section \ref{sec: limitation of regularity} we provide an example where 
the class of $C^1$ functions along geodesics is limited to 
functions with critical points at points where the curvature of the space 
becomes unbounded.
\begin{lemma}[basic growth bound]\label{lm: bg bound}
    Let $f:G\to\mathbb{R}$ have local uniform Lipschitz gradient along geodesics with constant $L>0$ on the convex subset $U\subset G$. Then
    \begin{equation}\label{eq: bg bound}
        f((1-t)x\oplus ty)\leq f(x)+tdf(x)[\gamma_{[x,y]}]+\frac{L}{2}t^2d(x,y)^2\quad \forall x,y\in U, t\in[0,1].
    \end{equation}
    In particular, at $t=1$
    \begin{equation}\label{eq: L-smooth}
        f(y)-f(x)\leq df(x)[\gamma_{[x,y]}]+\frac{L}{2}d(x,y)^2 \qquad\text{for all}\,\,x,y\in U.
    \end{equation}
\end{lemma}
\begin{proof}
    Fix $x, y\in U$, and set $\gamma:=\gamma_{[x,y]}\in G_x$. Consider the function
    $$
    g:[0,1]\to\mathbb{R} \qquad t\mapsto f(\gamma(t))=f((1-t)x\oplus ty).
    $$
    Since $\gamma_{[y,x]}(t)=\gamma_{[x,y]}(1-t)$, the chain rule applied to $g:=f\circ\gamma$ yields $\frac{\mathrm{d}}{\mathrm{d}t}g(1-t)=-g'(t)$ for all $t\in(0,1)$. Therefore, we conclude that
    \begin{align}
        Df(z)[\gamma_{[z,x]}]&=\frac{1}{d(x,y)}\frac{\mathrm{d}}{\mathrm{d}t}g(1-\lambda)=\frac{-g'(\lambda)}{d(x,y)}\notag\\
        &=\frac{-Df(z)[\gamma_{[z,y]}]d(x,y)}{d(x,y)}=-Df(z)[\gamma_{[z,y]}].\label{eq: dir grad inv dir}
    \end{align}
    It follows by \eqref{eq: L-gradient} that the function $h(t):=Df(\gamma_{[x,y]}(t))[\gamma_{[\gamma_{[x,y]}(t),y]}]d(x,y)$ is finite and Lipschitz continuous with constant $Ld(x,y)^2$ on $(0,1)$. Moreover, $h(t)=(f\circ\gamma_{[x,y]})'(t)$. By Rademacher's theorem \cite{Rademacher1917} the function $h$ is almost everywhere differentiable on $[0,1]$, yielding the Taylor expansion
    \begin{equation}\label{e:strong upper gradient}
        \begin{aligned}
            g(t)&=g(0)+tg'(0)+\int_0^t(t-\tau)\lim_{s\to\tau}\frac{\bigl|g'(s)-g'(\tau)\bigr|}{|s-\tau|}\,\mathrm{d}\tau\\
            &\leq g(0)+tdf(x)[\gamma_{[x,y]}]+\int_0^t(t-\tau)\lim_{s\to\tau}\frac{L|s-\tau|d(x,y)^2}{|s-\tau|}\,\mathrm{d}\tau\\
            &=g(0)+tdf(x)[\gamma_{[x,y]}]+\frac{L}{2}t^2d(x,y)^2.
        \end{aligned}     
    \end{equation}
    The assertion follows.
\end{proof}
The inequality \eqref{e:strong upper gradient} reveals the similarity and differences of our construction with the {\em strong upper gradients} defined in \cite[Definition 1.2.1]{AmbGigSav2005} (see also \cite[Eq.(2)]{Haj94} and \cite[Eq.(2.10)]{HeiKos98}). Firstly, the sign of the directional derivative is retained, so the inequality is characterizing something different than strong (upper) gradients. Secondly, the inequality is required only along geodesics.

A function $f:G\to\mathbb{R}$ satisfying \eqref{eq: L-smooth} will be said to be {\em locally $L$-smooth with neighborhood $U$ and constant $L$}. The continuity and $C^1$-regularity along geodesics of $f$ is essential, because without it functions can increase in value on arbitrarily small neighborhoods in directions of descent. This is demonstrated in the next example.
\begin{example}[a function can increase in descent directions]
    Let $G=\mathbb{R}^2$ be equipped with the Euclidean distance. Consider the function $f:\mathbb{R}^2\to\mathbb{R}$ defined by
    $$
    f(x_1,x_2)=\begin{cases}
        x_1\,,&\quad\text{if}\,\,x_2=0\\
        \\
        -x_1\,,&\quad\text{else.}
    \end{cases}
    $$
    Take $x:=(0,0)^{\mathrm{T}}$, $y:=(1,0)^{\mathrm{T}}$ and $w^k:=(1,k^{-1})^{\mathrm{T}}$. It is clear that the subslope $S^-f(x)[\gamma_{[x,y]}]=1$, while $S^-f(x)[\gamma_{[x,w^k]}]=-1$. Actually, since $w^k\to y$ we have $d^-f(x)[\gamma_{[x,y]}]=-1$ and $\gamma_{[x,y]}$ is a descent direction of $f$ at $x$. However, $f(x+\lambda y)>f(x)$ for all $\lambda\in(0,1]$.
\end{example}
\begin{theorem}[convergence with local uniform Lipschitz gradient]\label{th: BD with unif L-gradient}
    Let $f:G\to\bbR$, and let $(x^k)_{k\in\bbN}$ be a sequence generated by Algorithm \ref{alg: BD}.
    \begin{enumerate}
        \item[(i)] If $f$ has local uniform Lipschitz gradient along geodesics with constant $L>0$ on the convex set $U\supset\mathrm{lev}_{\leq f(x^0)}(f)$, then Assumption \ref{as: gen assm} \eqref{as: gen assm sd} holds for any descent direction $\gamma_{[x^k,y^k]}$ at all points $x^k\in\mathrm{lev}_{\leq f(x^0)}(f)$, $y^k\in U$, and Assumption \ref{as: gen assm} \eqref{as: gen assm mdfv} is satisfied for all step scalings $\lambda_k$ small enough.
        \item[(ii)] If in addition, $f$ satisfies Assumption \ref{as: gen assm} \eqref{as: gen assm bbb} and $\lambda_kd(x^k,y^k)$, the step lengths, satisfy Assumption \ref{as: gen assm} \eqref{as: gen assm ds}, then the sequence $(x^k)_{k\in\bbN}$ satisfies
        $$
        \lim_{k\to\infty}d(x^{k+1},x^k)=0,
        $$
        and at least one of the following must happen:
        \begin{enumerate}
            \item[(a)] the sequence $(x^k)_{k\in\bbN}$ converges finitely to a point $\xbar\in U$ satisfying
            $$
            d^-f(\xbar)[\gamma_{[\xbar, y]}]\geq 0 \qquad\text{for all}\,\,y\in U;
            $$
            \item[(b)] the sequence of directional subderivatives converges to zero, i.e.,
            $$
            \lim_{k\to\infty} d^-f(x^k)[\gamma_{[x^k, y^k]}]=0.
            $$
        \end{enumerate}
    \end{enumerate}
\end{theorem}
\begin{proof}
    \begin{enumerate}
        \item[(i)] We first note that for $f$ directionally differentiable at $x\in U$ and $\gamma:=\gamma_{[x,y]}\in G_x$ with $y\in U$ a descent direction, i.e. $df(x)[\gamma]<0$, for any $t\in(0,1]$ small enough $f(\gamma(t))-f(x)=f((1-t)x\oplus ty)-f(x)<0$. So the sequence $(f(x^k))_{k\in\bbN}$ generated by Algorithm \ref{alg: BD} is monotonically decreasing, and every $x^k$ remains in $\mathrm{lev}_{\leq f(x^0)}(f)$, for all step scalings $\lambda_k$ small enough. The basic growth bound Lemma \ref{lm: bg bound} yields for $\sigma\in(0,1]$
        $$
        f((1-t)x\oplus ty)\leq f(x)+t\sigma df(x)[\gamma_{[x,y]}]+\frac{L}{2}t^2d(x,y)^2
        $$
        at any point $x\in\mathrm{lev}_{\leq f(x^0)}(f)$ for any descent direction $\gamma_{[x,y]}$ with $y\in U$ and $t$ small enough. In particular, if the sequence $(x^k)_{k\in\mathbb{N}}$ remains in $\mathrm{lev}_{\leq f(x^0)}(f)$, then Assumption \ref{as: gen assm} \eqref{as: gen assm sd} is satisfied with $M=L/2$.
        \item[(ii)] This part follows from Lemma \ref{lm: proto descent}, since local uniform Lipschitz gradient along geodesics is a sufficient condition for
        Assumption \ref{as: gen assm} \eqref{as: gen assm mdfv} and \eqref{as: gen assm sd}.
    \end{enumerate}
\end{proof}
The familiar characterization of minima for continuously differentiable functions on a Euclidean space can be derived for a function $f:G\to\mathbb{R}$ provided that $f$ has local uniform Lipschitz gradient along geodesics. Before stating this claim, we first observe that for $x\in\mathrm{gi}\,(G)$ and $y\in G$ we can extend the geodesic $\gamma_{[x,y]}$ to a new geodesic $\widetilde{\gamma}:[-\varepsilon,1]\to G$ for some positive $\varepsilon>0$ such that $\widetilde{\gamma}(0)=x$. We then rescale $\widetilde{\gamma}$ to $\widehat{\gamma}:[0,1]\to G$ such that $\widehat{\gamma}(t_0)=x$ for some $t_0>0$ sufficiently small. The same holds on a convex subset $U\subset G$.
\begin{proposition}\label{pr: zero dervt}
    Let $f:G\to\mathbb{R}$ with $\mathrm{argmin}_G(f)\subset\mathrm{gi}\,(G)$ have local uniform Lipschitz gradient along geodesics with constant $L>0$ on the convex neighborhood $U\subset G$ where $\mathrm{argmin}_G(f)\subset\mathrm{gi}\,(U)$. If $\xbar\in\mathrm{argmin}_G(f)$ then
    \begin{equation}\label{eq: zero dervt}
        df(\xbar)[\gamma_{[\xbar,y]}]=0 \qquad\text{for all}\,\,y\in U.
    \end{equation}
    Conversely, if $f$ is convex and \eqref{eq: zero dervt} holds, then $\xbar\in\mathrm{argmin}_G(f)$.
\end{proposition}
\begin{proof}
    Let $\xbar\in\mathrm{argmin}_G(f)$ and $y\in U$. Extend and rescale the geodesic $\gamma_{[\xbar,y]}$ to $\widehat{\gamma}:[0,1]\to U$ such that $\widehat{\gamma}(t_0)=\xbar$ for some $t_0>0$ sufficiently small. Define the function $\widehat{f}:[0,1]\to\mathbb{R}$, $t\mapsto f(\widehat{\gamma}(t))$. Since $f$ has uniform Lipschitz gradient, the composed function $\widehat{f}$ is $C^1$ on the open interval $(0,1)$. Note that $t_0\in(0,1)$ is a global minimizer of $\widehat{f}$ by construction. Proposition \ref{pr: C^1 implies diff} shows that the function $\widehat{h}:(0,1)\to\mathbb{R}$, $t\mapsto df(\widehat{\gamma}(t))[\gamma_{[\widehat{\gamma}(t),y]}]/(1-t)$ satisfies $\widehat{h}(t)=\widehat{f}'(t)$ for all $t\in(0,1)$. Thus it is necessary that $\widehat{h}(t_0)=0$. Hence also $df(\xbar)[\gamma_{[\xbar,y]}]=0$.

    The converse implication follows from Theorem \ref{th: FNC-P/FSC}.
\end{proof}
\subsection{Steepness and the Polyak-Łojasiewicz property}\label{s:max slope PL}
This subsection introduces two key properties for the proof of convergence of the BD method: the {\em Polyak-Łojasiewicz (PŁ) property} and the {\em sufficient decrease condition}.
\begin{definition}[steepness]\label{d:steepness}
    Let $f:G\to\bbR$ and $x\in G$. The {\em steepness} of $f$ at $x$ is defined by
    $$
    |D|f(x):=\sup_{\gamma\in G_x^{\circ}}\bigl|D^+f(x)[\gamma]\bigr|.
    $$
\end{definition}
Note that the steepness can be infinite, and as a mapping on $G$ to the extended reals, $|D|f(\cdot)$ is defined everywhere on $G$. A similar construction that is prevalent is called the {\em descending slope} in \cite[Definition 1.1]{DeGMarTos80} and the {\em absolute gradient} in \cite[Definition 3.1]{OhtaPal15}.  Steepness, in contrast, does not distinguish between ascent or descent. When $(G,d)$ is the Euclidean space $(\mathbb{R}^n,\|\cdot\|)$ and $f:\mathbb{R}^n\to\mathbb{R}$ is differentiable, $|D|f(x)=\|\nabla f(x)\|$.

Steepness is key to a type of Poincaré inequality that is a standard assumption in many applications. In this context see \cite[Eq.(4.3)]{Cheeger99}. Amongst the Poincaré-type inequalities is the Łojasiewicz inequality \cite{Loj63}, and a special case attributed to Polyak \cite{Polyak63}. A function $f:G\to\bbR$ is said to satisfy the Polyak-Łojasiewicz (PŁ) property at a point $x^*$ whenever
\begin{equation}\label{eq: PL}
    \exists\,C>0\mbox{ and }U\subset G~:~  |D|f(x)^{\frac{2}{p-1}}\geq C\Bigl(f(x)-f(x^*)\Bigr) \qquad\text{for all}\,\,x\in U.
\end{equation}
It is well known that in the Euclidean setting that strong convexity implies the PŁ property. In the metric setting, in contrast, both strong convexity and local uniform Lipschitz gradients along geodesics are required. Throughout our development the reference point $x^*$ will always be a local minimum, characterized by $x^*\in \mathrm{argmin}_U(f)$, where $U$ is a convex set containing $x^*$ on its interior. When $U=G$ the point $x^*$ is a global minimum and we denote by $f_{\min}:=\min_G(f)$ the global minimum value of the function $f$.

We collect the assumptions for Algorithm \ref{alg: BD} into one place and we will specify in each result which of these are used.
\begin{assumption}[BD assumptions]\label{as: BD assm}
    Let $(x^k)$ be a sequence generated by Algorithm \ref{alg: BD} initialized with $x^0\in U$, where $U\subset G$ is convex.
    \begin{enumerate}[(a)]
        \item\label{as: BD assm argmin} The function $f$ is bounded below on $U$ and $\mathrm{argmin}_U(f)\neq\emptyset$.
        \item\label{as: BD assm mu-cvx} The function $f$ is strongly convex on $U$ with constant $\mu>0$;
        \item\label{as: BD assm L-gradient} The function $f$ has local uniform Lipschitz gradient along geodesics with constant $L>0$ on $U$.
        \item\label{as: BD assm suff descent} The sequence $(x^k)$ satisfies the sufficient decrease condition \eqref{eq: BD assm suff descent}
        \item\label{as: BD assm PL} The function $f$ satisfies the PŁ property at $x^*\in \mathrm{argmin}_U(f)$ with constant $C_U>0$ on the neighborhood $U$.
        \item\label{as: BD assm sl} The sequence $(x^k)$ satisfies
        \begin{equation}\label{eq: BD assm sl}
            \exists\,\delta>0 ~:~ d(x^k,x^{k+1})^2\leq\frac{1}{\delta}|D|f(x^k)^{\frac{2}{p-1}} \qquad\text{for all}\,\,k\in\bbN.
        \end{equation}
    \end{enumerate}
\end{assumption}
The next lemma is very useful in the derivation of PŁ property and convergence analysis.  The proof follows by straightforward application of the definitions.
\begin{lemma}[directional derivative bound]\label{lm: BD dir-dervt bound}
    Let $f:G\to\bbR$ have local uniform Lipschitz gradient along geodesics on $U$. Then
    \begin{equation}\label{eq: BD dir-dervt bound}
        \bigl|df(x)[\gamma_{[x,y]}]\bigr|\leq|D|f(x)d(x,y) \qquad\text{for all}\,\,x,y\in U\,\,\text{and}\,\,x\neq y.
    \end{equation}
\end{lemma}
We are now in a position to state sufficient conditions for PŁ property on uniformly convex metric spaces.
\begin{proposition}[sufficient conditions for the PŁ property]\label{pr: suff cond for PL}
    Let $f:G\to\bbR$ satisfy Assumption \ref{as: BD assm} (\ref{as: BD assm argmin})-(\ref{as: BD assm L-gradient}). Furthermore, let $U\subset G$ be a convex neighborhood of a point $x^*\in \mathrm{argmin}_U(f)$ where $\mathrm{argmin}_U(f)\subset\mathrm{gi}\,(U)$. Then $f$ satisfies the PŁ property at $x^*$ on $U$ with constant
    $$
    C_U:=\frac{2}{L}\left(\frac{\mu}{2}\right)^{\frac{2}{p-1}}>0.
    $$
\end{proposition}
\begin{proof}
    Since $f$ is strongly convex with constant $\mu>0$, and has local uniform Lipschitz gradient along geodesics with constant $L>0$ on $U$, it follows that
    $$
    f(y)-f(x)\geq df(x)[\gamma_{[x,y]}]+\frac{\mu}{2}d(x,y)^p \qquad\text{for all}\,\,x,y\in U.
    $$
    Taking $y$ to be $x^*\in\mathrm{argmin}_U(f)$ and denoting $f_* = f(x^*)$ yields for all $x\in U$
    $$
    f_*-f(x)\geq df(x)[\gamma_{[x,x^*]}]+\frac{\mu}{2}d(x,x^*)^p.
    $$
    Rearranging terms yields
    $$
    -df(x)[\gamma_{[x, x^*]}]\geq f(x)-f_*+\frac{\mu}{2}d(x,x^*)^p\geq\frac{\mu}{2}d(x,x^*)^p
    $$
    which together with \eqref{eq: BD dir-dervt bound} gives the estimate
    \begin{equation}\label{eq: str-cvx implies PL 2}
        \frac{\mu}{2}d(x,x^*)^{p-1}\leq|D|f(x) \qquad\text{for all}\,\,x\in U.
    \end{equation}
    By $L$-smoothness, for all $x\in U$
    $$
    f(x)-f_*\leq df(x^*)[\gamma_{[x^*,x]}]+\frac{L}{2}d(x^*,x)^2,
    $$
    which together with \eqref{eq: str-cvx implies PL 2} implies
    $$
    f(x)-f_*\leq df(x^*)[\gamma_{[x^*,x]}]+\frac{L}{2}\left(\frac{2}{\mu}|D|f(x)\right)^{\frac{2}{p-1}} \qquad\text{for all}\,\,x\in U.
    $$
    The directional derivative $df(x^*)[\gamma_{[x^*,x]}]$ vanishes due to Proposition \ref{pr: zero dervt}, and the assertion follows from the resulting inequality.
\end{proof}
\subsection{Convergence of BD method}\label{s:convergence}
We now apply the PŁ property and the sufficient decrease condition to obtain convergence of function values and iterates to a local minimum.
\begin{theorem}[convergence of function values with PŁ and sufficient descent]\label{th: cvg func value}
    For $f:G\to\mathbb{R}$ and $U\subset G$ closed and convex, let Assumption \ref{as: BD assm} (\ref{as: BD assm argmin}),(\ref{as: BD assm suff descent}), and (\ref{as: BD assm PL}) hold for the sequence $(x^k)_{k\in\bbN}$ generated by Algorithm \ref{alg: BD} with respective constants $\delta$ and $C_U>0$. Then the function values of the iterates satisfy
    \begin{equation}\label{eq: cvg func value}
        f(x^k)-f_*\leq\Bigl(f(x^0)-f_*\Bigr)\left(1-\frac{C_U}{\delta}\right)^k \qquad\text{for all}\,\,k,
    \end{equation}
    where $f_*:=\min_U(f)$. If, in addition $C_U\in(0,\delta)$, then the sequence of function values converges linearly to the minimum value of $f$ on $U$ with rate given by \eqref{eq: cvg func value}.
\end{theorem}
\begin{proof}
    A combination of \eqref{eq: BD assm suff descent} and PŁ property yields for all $k\in\bbN$
    $$
    f(x^{k+1})-f(x^k)\leq-\frac{1}{\delta}|D|f(x^k)^{\frac{2}{p-1}}\leq-\frac{C_U}{\delta}\Bigl(f(x^k)-f_*\Bigr).
    $$
    This is equivalent to
    $$
    f(x^{k+1})-f_*\leq\left(1-\frac{C_U}{\delta}\right)\Bigl(f(x^k)-f_*\Bigr) \qquad\text{for all}\,\,k\in\bbN.
    $$
    Iterating this inequality implies the claim.
\end{proof}
\begin{theorem}[convergence of iterate sequence with PŁ property]\label{th: cvg ite}
    Let $(G,d)$ be a complete metric space. For the function $f:G\to\mathbb{R}$, and $U\subset G$ convex, let Assumption \ref{as: BD assm} (\ref{as: BD assm argmin}),(\ref{as: BD assm L-gradient})-(\ref{as: BD assm sl}) hold -- with respective constants $\delta, C_U>0$ -- for the sequence $(x^k)_{k\in\bbN}$ generated by Algorithm \ref{alg: BD}. In addition, let $C_U\in(0,\delta)$. Then the sequence converges to $\xbar\in U$, a minimum of $f$ on $U$, with $R$-linear rate given by
    \begin{equation}\label{eq: cvg ite}
        d(x^k,\xbar)\leq\frac{\sqrt{f(x^0)-f(\xbar)}}{1-\sqrt{1-\dfrac{C_U}{\delta}}}\left(\sqrt{1-\frac{C_U}{\delta}}\right)^{k} \qquad\text{for all}\,\,k\in\Nbb.
    \end{equation}
\end{theorem}
\begin{proof}
    By Assumption Assumption \ref{as: BD assm} (\ref{as: BD assm argmin}) there is an $x^*\in\argmin_U f$ with $f_*=f(x^*)$.
    Combining Assumption \ref{as: BD assm} (\ref{as: BD assm sl}) and \eqref{eq: BD assm suff descent} implies for all $k$
    $$
    d(x^{k+1},x^k)^2\leq\frac{1}{\delta}|D|f(x^k)^{\frac{2}{p-1}}\leq f(x^k)-f(x^{k+1}).
    $$
    Since $f(x^{k+1})\geq f_*$, Theorem \ref{th: cvg func value} yields
    $$
    d(x^{k+1},x^k)^2\leq f(x^k)-f_*\leq\left(1-\frac{C_U}{\delta}\right)^k\Bigl(f(x^0)-f_*\Bigr).
    $$
    Let $c:=\sqrt{f(x^0)-f_*}$ and $w:=1-C_U/\delta$. By the triangle inequality, for any $m\in\bbN$ and $n\in\bbN$,
    $$
    d(x^{m+n},x^m)\leq\sum_{j=m}^{m+n-1}d(x^{i+1},x^i)\leq c\sum_{j=m}^{m+n-1}\sqrt{w}^j.
    $$
    Since $C_U\in(0,\delta)$ we conclude that $w\in(0,1)$ and
    \begin{equation}\label{eq: cvg ite 1}
        d(x^{m+n},x^m)\leq c\sum_{j=m}^{+\infty}\sqrt{w}^j=\frac{c\sqrt{w}^m}{1-\sqrt{w}},
    \end{equation}
    hence the sequence $(x^k)_{k\in\bbN}$ is a Cauchy sequence, and thus convergent by the completeness of $(U,d)$. Let $\xbar\in U$ be the limit of this sequence. Taking the limit as $n\to\infty$ in \eqref{eq: cvg ite 1} gives $d(x^m, \xbar)\leq\frac{c\sqrt{w}^m}{1-\sqrt{w}}$, completing the proof.
\end{proof}
We emphasize that the completeness of the metric space $(G,d)$ (and hence the smaller metric space $(U, d)$) guarantees the existence of the limiting point $\xbar\in U$. However, if the function $f:G\to\mathbb{R}$ is strongly convex on $U$, then even without the completeness of the space $(U,d)$, linear convergence can still be obtained.
\begin{corollary}\label{cl: cvg ite for strongly cvx func}
    For the function $f:G\to\mathbb{R}$, and $U\subset G$ convex,
    let Assumption \ref{as: BD assm} (\ref{as: BD assm argmin}),(\ref{as: BD assm mu-cvx}),(\ref{as: BD assm suff descent}), and (\ref{as: BD assm PL}), hold -- with respective constants $\mu, \delta, C_U>0$ --  for the sequence $(x^k)_{k\in\bbN}$ generated by Algorithm \ref{alg: BD}.  Then  $\mathrm{argmin}_U(f)=\{\bar{x}\}$, and for $f_*:=f(\xbar)$
    \begin{equation}\label{eq: cvg ite for strongly cvx func}
        d(x^k,\bar{x})\leq\sqrt[p]{\frac{2}{\mu}\Bigl(f(x^0)-f_*\Bigr)}\left(\sqrt[p]{1-\frac{C_U}{\delta}}\right)^k \qquad\text{for all}\,\,k\in\Nbb.
    \end{equation}
    If in addition the constant $C_U$ satisfies $C_U\in(0,\delta)$, then the sequence of iterates $(x_k)_{k\in\bbN}$ converges $R$-linearly to $\bar{x}$ with rate given by \eqref{eq: cvg ite for strongly cvx func}.
\end{corollary}
\begin{proof}
    By Assumption Assumption \ref{as: BD assm} (\ref{as: BD assm argmin}) there is an $\xbar\in\argmin_U f$ with $f_*:=f(\xbar)=\min_U f$. Since $f$ is strongly convex, then by Theorem \ref{th: FNC-P/FSC}, $f$ has a unique minimizer on $U$, $\mathrm{argmin}_U(f)=\{\bar{x}\}$, and $df(\bar{x})[\gamma_{[\bar{x},y]}]\geq 0$ for all $y\in U$. A combination of Proposition \ref{pr: dir dervt of cvx fct} and Theorem \ref{th: cvg func value} thus implies
    \begin{align*}
        \frac{\mu}{2}d(x^k,\bar{x})^p&\leq f(x^k)-f(\bar{x})-df(\bar{x})[\gamma_{[\bar{x},x^k]}]\leq\left(1-\frac{C_U}{\delta}\right)^k\Bigl(f(x^0)-f_*\Bigr),
    \end{align*}
    from which the result follows.
\end{proof}
\subsection{Proofs of main results}\label{sec: proofs main results}
We are now ready to prove the main results. They follow mainly from Lemma \ref{lm: proto descent}, Theorem \ref{th: BD with unif L-gradient}, Theorem \ref{th: cvg func value} and Theorem \ref{th: cvg ite}.

{\em Proof of Theorem \ref{th: main summary 1}.} By Theorem \ref{th: BD with unif L-gradient}(i), the assumption of local uniform Lipschitz gradient along geodesics implies that assumptions (c)-(d) of Lemma \ref{lm: proto descent} hold for all step scalings $\lambda_k$ small enough. Then, as with part (ii) of Theorem \ref{th: BD with unif L-gradient}, assumptions (a) and (c) of Theorem \ref{th: main summary 1} permit application of Lemma \ref{lm: proto descent} to conclude that at least one of (i) and (ii) must happen.

In the event of case (i), Theorem \ref{th: BD with unif L-gradient} shows that $\bar{x}\in G$ satisfies the first-order necessary condition
$$
d^-f(\bar{x})[\gamma_{[\bar{x}, y]}]\geq 0 \qquad\text{for all}\,\,y\in G.
$$
If in addition $f$ is convex, this condition is also sufficient by Theorem \ref{th: FNC-P/FSC} and this point $\bar{x}\in G$ is a global minimizer of $f$. \qed

{\em Proof of Theorem \ref{th: main summary 2}.} With the sufficient decrease condition \eqref{eq: BD assm suff descent} satisfied with constant $\delta>0$, and the PŁ property satisfied with constant $C_U\in(0,\delta)$, Theorem \ref{th: cvg func value} establishes that $f(x^k)\to f_*$ with
$$
f(x^k)-f_*\leq\left(1-\frac{C_U}{\delta}\right)^k\Bigl(f(x^0)-f_*\Bigr) \qquad\text{for all}\,\,k.
$$
Note that this is independent of the convergence of $(x^k)_{k\in\bbN}$.

Next, taking into account the additional assumptions that $(G,d)$ is complete, $f$ has local uniform Lipschitz gradient along geodesics with constant $L>0$ on $U$, and the sequence satisfies \eqref{e:suff iterate descent}, then Theorem \ref{th: cvg ite} establishes that the sequence $(x^k)_{k\in\bbN}$ converges to some point $\bar{x}\in U$ with at least $R$-linear rate given by
$$
d(x^k,\bar{x})\leq\frac{\sqrt{f(x^0)-f(\xbar)}}{1-\sqrt{1-\frac{C_U}{\delta}}}\left(\sqrt{1-\frac{C_U}{\delta}}\right)^{k} \qquad\text{for all}\,\,k\in\Nbb.
$$
The continuity of $f$ implies $f(x^k)\to f(\bar{x})$ and by the uniqueness of the limit, we have $f(\bar{x})=\min_U f$, i.e. $\bar{x}\in\mathrm{argmin}_U(f)$ is a (local) minimizer of $f$ on $U$. \qed

As in Corollary \ref{cl: cvg ite for strongly cvx func}, if $f:G\to\mathbb{R}$ is strongly convex on $U$ with constant $\mu>0$, then the completeness assumption of $(G,d)$ can be dropped, and $\mathrm{argmin}_U(f)=\{\bar{x}\}$ with
$$
d(x^k,\bar{x})\leq\sqrt[p]{\frac{2}{\mu}\Bigl(f(x^0)-f(\xbar)\Bigr)}\left(\sqrt[p]{1-\frac{C_U}{\delta}}\right)^k \qquad\text{for all}\,\,k\in\Nbb.
$$
Proposition \ref{pr: suff cond for PL} establishes the following corollary.
\begin{corollary}
    The conclusion of Theorem \ref{th: main summary 2} holds if the assumption (c) that the P\L  property holds is replaced by the following:
    \begin{enumerate}
        \item[(c')] the initial point $x^0\in U\subset G$ satisfies $\mathrm{lev}_{\leq f(x^0)}(f)\subset\mathrm{gi}\,(U)$, and $f$ is strongly convex on $U$
        with constant $\mu>0$ satisfying
        $$
        C_U:=\frac{2}{L}\left(\frac{\mu}{2}\right)^{\frac{2}{p-1}}\in(0,\delta).
        $$
    \end{enumerate}
\end{corollary}
\section{Steepest and almost steepest descent}\label{sec: SD method}
We show next that the steepest descent direction, appropriately defined in the right context, exists and provides a direction in which all of the assumptions for convergence can be satisfied. We first state a generic condition for the sufficient descent.
\begin{proposition}\label{pr: suff cond for SD}
    Let $f:G\to\mathbb{R}$ have local uniform Lipschitz gradient along geodesics with constant $L>0$ on the convex subset $U$. Let $x\in U$ and $x^+:=\gamma_{[x,y]}(\lambda)$ for some
    $$
    y\in\mathcal{D}(x):=\bigl\{ y\in U\setminus\{x\} \,\big|\, Df(x)[\gamma_{[x,y]}]<0 \bigr\}.
    $$
    For a given $\delta>0$, let the point $y$ and the step scaling $\lambda\in(0,1]$ satisfy
    \begin{equation}\label{eq: sd condition}
        \alpha(x,y):=-\left(\frac{2+L}{2\delta}\right)\frac{|D|f(x)^{\frac{2}{p-1}}}{df(x)[\gamma_{[x,y]}]}\leq \lambda\leq\beta(x,y):=\frac{|D|f(x)^{\frac{1}{p-1}}}{\sqrt{\delta}d(x,y)}.
    \end{equation}
    Then
    $$
    f(x^+)-f(x)\leq-\frac{1}{\delta}|D|f(x)^{\frac{2}{p-1}}.
    $$
\end{proposition}
\begin{proof}
    By the $L$-smoothness and \eqref{eq: sd condition},
    \begin{align*}
        f(x^+)-f(x)&\leq df(x)[\gamma_{[x,x^+]}]+\frac{L}{2}d(x,x^+)^2\\
        &\leq df(x)[\gamma_{[x,x^+]}]+\frac{L}{2\delta}|D|f(x)^{\frac{2}{p-1}}\\
        &\leq-\left(\frac{2+L}{2\delta}\right)|D|f(x)^{\frac{2}{p-1}}+\frac{L}{2\delta}|D|f(x)^{\frac{2}{p-1}}=-\frac{1}{\delta}|D|f(x)^{\frac{2}{p-1}}.
    \end{align*}
\end{proof}
The choice of step scaling $\lambda$ is {\em feasible} when $\alpha(x,y)\leq\beta(x,y)$, which is equivalent to
\begin{equation}\label{eq: sd condition eq}
    Df(x)[\gamma_{[x,y]}]\leq-\left(\frac{2+L}{2\sqrt{\delta}}\right)|D|f(x)^{\frac{1}{p-1}}.
\end{equation}
The step is {\em admissible} as long as $\alpha(x,y)\leq 1$.
We denote by $I_{\mathrm{fea}}:=[\alpha,\beta]$ the interval of feasible step scalings, and by $I_{\mathrm{ad}}:=(0,1]\cap I_{\mathrm{fea}}$ the interval of admissible step scalings. We show next that a {\em steepest } descent step $y\in\mathcal{D}(x)$ always permits admissible step scalings.
\subsection{The steepest descent direction}\label{s:steepest descent}
We now provide guarantees in a metric setting for attainment of the infimum of normalized directional derivatives, and show that such a descent step provides admissible step scalings.
\begin{lemma}\label{lem: core property}
    Let $x\in \mathrm{gi}\,(G)$ and $\gamma_{[x,y]}\in G_x^{\circ}$, and let $f:G\to\mathbb{R}$ be directionally differentiable in the direction $\gamma_{[x,y]}$.
    For each constant $\rho>0$ small enough, there exists a geodesic $\eta\in G_x^{\circ}$ with either $\eta(t)=y$
    or $\eta(1)=\gamma_{[x,y]}(t)$ for some $t\in (0,1]$, satisfying
    $$
    Df(x)[\eta]=Df(x)[\gamma_{[x,y]}] \quad\text{and}\quad d(\eta(0),\eta(1))=\rho.
    $$
\end{lemma}
\begin{proof}
    By the definition of the geodesic interior, for any $y\in G\setminus\{x\}$ there exists a geodesic $\eta'$ with length $\rho'$ and starting point  $\eta'(0)=y$ that contains the points $\gamma_{[x,y]}(t)$ on its interior for all $t\in [0,1]$. The claim is proved by constructing a geodesic $\eta$ that reverses the direction of $\eta'$, shifts the start to the point $x$, and rescales it to length $\rho\leq \rho'$: $\eta(0)=x$, $d(\eta(0), \eta(1))=\rho\leq \rho'$, and either $y=\eta(t)$ or $\eta(1)=\gamma_{[x,y]}(t)$ for some $t\in (0,1]$. In either case, it follows from Remark \ref{rm: truncation of geod} that $Df(x)[\eta]=Df(x)[\gamma_{[x,y]}]$.
\end{proof}
With this result, we are able to show the existence of the steepest descent direction. As in the Euclidean case, the steepest descent direction can be found by searching over spheres of a fixed radius:
$$
\mathbb{S}^{G_x}_{\rho}:=\bigl\{ \gamma\in G_x \,\big|\, d(\gamma(0),\gamma(1))=\rho \bigr\}.
$$
In comparison to \cite[Theorem 3.2]{DeGMarTos80} in the context of gradient flows, our proof and statement are constructive.
\begin{proposition}[existence of the steepest descent step]\label{pr: ex of SD direction}
    Let $(G,d)$ be boundedly compact,  let $f:G\to\mathbb{R}$ have local uniform Lipschitz gradient along geodesics with constant $L>0$ on the convex subset $U$, let $x\in \mathrm{gi}\,(U)$, and let $\rho>0$ be small enough that $\mathbb{S}^{G_x}_{\rho}\subset \mathrm{gi}\,(U)$. Then the function $Df(x)[\,\cdot\,]$ defined on $G_x^{\circ}$ is continuous, bounded and there exist $\gamma_{\max},\gamma_{\min}\in G_x^{\circ}$ and $y_{\max}, y_{\min}\in\mathbb{S}_{\rho}(x):=\bigl\{ z\in G \,\big|\, d(x,z)=\rho \bigr\}$ such that
    \begin{equation}\label{eq:sa-sd}
    \begin{aligned}
        Df(x)[\gamma_{\max}] &= Df(x)[\gamma_{[x,y_{\max}]}] = \max_{\gamma\in \mathbb{S}^{G_x}_{\rho}}Df(x)[\gamma],\quad\mbox{and}\\
        Df(x)[\gamma_{\min}]&= Df(x)[\gamma_{[x,y_{\min}]}] = \min_{\gamma\in \mathbb{S}^{G_x}_{\rho}}Df(x)[\gamma].
        \end{aligned}
    \end{equation}
    Moreover,
    \begin{equation}\label{eq: steepest descent}
    \mbox{(a)}\, Df(x)[\gamma_{[x,y_{\max}]}]=|D|f(x), \quad\mbox{and}\quad \mbox{(b)}\, Df(x)[\gamma_{[x,y_{\min}]}]=-|D|f(x).
    \end{equation}
\end{proposition}
\begin{proof}
    The proof is constructive. Fix some $\gamma\in\mathbb{S}^{G_x}_{\rho}\subset U$ and let $(\eta^k)_{k\in\bbN}$ be a sequence in $\mathbb{S}^{G_x}_{\rho}$ such that $\eta^k\to\gamma$. Then
    $$
    \lim_{k\to\infty}df(x)[\eta^k]=\lim_{k\to\infty}S^+f(x)[\eta^k]\leq\limsup_{\eta\to\gamma}S^+f(x)[\eta]=d^+f(x)[\gamma]=df(x)[\gamma],
    $$
    and analogously
    $$
    \lim_{k\to\infty}df(x)[\eta^k]=\lim_{k\to\infty}S^-f(x)[\eta^k]\geq\liminf_{\eta\to\gamma}S^-f(x)[\eta]=d^-f(x)[\gamma]=df(x)[\gamma].
    $$
    Therefore $\lim_{k\to\infty}df(x)[\eta^k]=df(x)[\gamma]$. Furthermore,
    $$
    d(\eta^k(0), \eta^k(1))=d(\gamma(0), \gamma(1))=\rho>0\qquad\text{for all}\,\,k\in\bbN,
    $$
    hence
    $$
    \lim_{k\to\infty}Df(x)[\eta^k]=\lim_{k\to\infty}\frac{df(x)[\eta^k]}{d(\eta^k(0), \eta^k(1))}=\frac{df(x)[\gamma]}{d(\gamma(0), \gamma(1))}=Df(x)[\gamma].
    $$
    Thus $Df(x)[\,\cdot\,]$ is continuous on $\mathbb{S}^{G_x}_{\rho}$. Since $(G,d)$ is boundedly compact, so is $(G_x,d_x)$. The set $\mathbb{S}^{G_x}_{\rho}$ is nonempty, bounded and closed in $G_x$, therefore $Df(x)[\,\cdot\,]$ attains its maximum and minimum on $\mathbb{S}^{G_x}_{\rho}$, verifying \eqref{eq:sa-sd}.

    Now, for any $y\in U$, let $z\in \mathrm{gi}\,(U)$ be any point obtained by the positive extension of the geodesic $\gamma_{[y,x]}$ beyond the endpoint $x$ (possible because $x\in \mathrm{gi}(U)$), and construct the corresponding geodesic extending from $x$, $\gamma_{[x,z]}$. For any geodesic constructed in this way \eqref{eq: dir grad inv dir} yields $Df(x)[\gamma_{[x,z]}]=-Df(x)[\gamma_{[x,y]}]$. In particular, let $\gamma_{\max}$ be a maximum of $Df(x)[\cdot]$ on $\mathbb{S}^{G_x}_{\rho}$. Lemma \ref{lem: core property} yields $Df(x)[\mathbb{S}^{G_x}_{\rho}]=Df(x)[G_x^{\circ}]$, and hence
    $$
    Df(x)[\gamma_{\max}]=\max_{\gamma\in \mathbb{S}^{G_x}_{\rho}}Df(x)[\gamma]=\sup_{\gamma\in G_x^\circ}Df(x)[\gamma].
    $$
    The last equality holds by Proposition \ref{pr: C^1 implies diff}. Now define $y=\gamma_{\max}(1)$. Then, for all $\gamma_{\min}\in G_x^\circ$ such that $\gamma_{\min}(1)=z\in\mathrm{gi}\,(G)$ is any point obtained by the positive extension of the geodesic $\gamma_{[y,x]}$ (the reverse of $\gamma_{\max}$) beyond the endpoint $x$, Lemma \ref{lem: core property} and \eqref{eq: dir grad inv dir} together yield
    $$
    \inf_{\gamma\in G_x^\circ}Df(x)[\gamma] = \min_{\gamma\in \mathbb{S}^{G_x}_{\rho}}Df(x)[\gamma] = Df(x)[\gamma_{\min}]=-\sup_{\gamma\in G_x^\circ}Df(x)[\gamma].
    $$
    Again, the first equality holds by Proposition \ref{pr: C^1 implies diff}. The assertion of the equivalence to the steepness follows directly from Definition \ref{d:steepness}.
\end{proof}
The characterization \eqref{eq: steepest descent} of Proposition \ref{pr: ex of SD direction} leads to the natural terminology of the {\em steepest descent} direction at iterate $x^k$ being the direction $\gamma_{[x^k,y^k]}$ where $Df(x^k)[\gamma_{[x^k,y^k]}]=-|D|f(x^k)$. The steepest descent direction $\gamma_{\min}$ characterized by \eqref{eq: steepest descent}(b) is the solution to an optimization problem in Definition \ref{d:steepness}. As with the proximal mapping, this problem has the same form as the original optimization problem \eqref{eq: model prb} and cannot be expected to be any easier to solve, except in special cases.  An {\em approximate} steepest descent direction, however, suffices. This is taken up in the next section. 

The next proposition shows that when the steepest descent direction is taken a suitable constant $\delta>0$ can be chosen in the sufficient decrease condition \eqref{eq: BD assm suff descent} such that the set of admissible step scalings, $I_{\mathrm{ad}}$, is not empty.
\begin{proposition}\label{pr: SD method}
    Let $(G,d)$ be boundedly compact, $p$-uniformly convex with $p\in(1,2]$. Assume that
    \begin{enumerate}[(i)]
        \item\label{pr: SD method i} $f:G\to\mathbb{R}$ is level-bounded and has local uniform Lipschitz gradient along geodesics with constant $L>0$ on the convex subset $U$;
        \item\label{pr: SD method ii} $\mathrm{lev}_{\leq f(x^0)}(f)\subset \mathrm{gi}\,(U)$;
        \item\label{pr: SD method iii} for $x\in\mathrm{lev}_{\leq f(x^0)}(f)$ and $\rho$ any positive constant such that $B_\rho(x)\subset \mathrm{gi}\,(U)$, the point $y\in\mathcal{D}(x)$ satisfies \eqref{eq: steepest descent}(b).
    \end{enumerate}
    Then there exists a uniform constant $\underline{\delta} >0$ such that for all $\delta\geq\underline{\delta} $, the interval $I_{\mathrm{ad}}$ is nonempty, and for any $\lambda\in I_{\mathrm{ad}}$
    $$
    f(x^+)-f(x)\leq-\frac{1}{\delta}|D|f(x)^{\frac{2}{p-1}} \qquad \text{with}\,\,x^+:=\gamma_{[x,y]}(\lambda).
    $$
\end{proposition}
Before moving to the proof, note that assumption \eqref{pr: SD method ii} of Proposition \ref{pr: SD method} is typically replaced with a {\em coercivity} assumption, namely, that the value of the function $f$ diverges to infinity as the argument approaches the boundary of the set $U$.
\begin{proof}
    Since $V:=\mathrm{lev}_{\leq f(x^0)}(f)\subset \mathrm{gi}\,(U)$ and $f$ is level-bounded, by Proposition \ref{pr: ex of SD direction}, for any point $x\in V$, the steepness $|D|f(x)$ is finite, i.e. the function $|D|f(\,\cdot\,):V\to\mathbb{R}$, $x\mapsto|D|f(x)$ is pointwise finite. Since $(G,d)$ is boundedly compact, we conclude that there is a uniform upper bound
    $$
    |D|f(x)\leq\overline{D}<+\infty \qquad\text{for all}\,\,x\in V.
    $$
    Since $p\in(1,2]$ we have $1/(p-1)\geq 1$. Thus, we can fix a $\underline{\delta}>0$ sufficiently large such that
    \begin{equation}\label{eq: SD method delta 1}
        \begin{aligned}
            &\mbox{(a)}\quad  s\geq\left(\frac{2+L}{2\sqrt{\underline{\delta}}}\right)s^{\frac{1}{p-1}} \quad\text{for all}\,\,0\leq s\leq\overline{D}\\
            &\mbox{(b)}\quad \left(\frac{2+L}{2\rho\underline{\delta}}\right)\overline{D}^{\frac{2}{p-1}-1}\leq 1.
        \end{aligned}
    \end{equation}
    It is obvious that \eqref{eq: SD method delta 1} remains valid for all $\delta\geq\underline{\delta}$ once $\overline{D}$ is fixed.

    Let $\delta\geq\underline{\delta}$ and $y\in\mathcal{D}(x)$ be chosen such that \eqref{eq: steepest descent}(b) holds. Inequality \eqref{eq: SD method delta 1}(a) yields
    $$
    -|D|f(x)\leq-\left(\frac{2+L}{2\sqrt{\delta}}\right)|D|f(x)^{\frac{1}{p-1}},
    $$
    which implies $\alpha\leq\beta$ and $I_{\mathrm{fea}}\neq\emptyset$. Moreover, \eqref{eq: SD method delta 1}(b) implies
    $$
    \alpha=\left(\frac{2+L}{2\delta}\right)\frac{|D|f(x)^{\frac{2}{p-1}-1}}{d(x,y)}=\left(\frac{2+L}{2\rho\delta}\right)|D|f(x)^{\frac{2}{p-1}-1}\leq 1.
    $$
    Therefore, the interval of admissible step scalings $I_{\mathrm{ad}}:=(0,1]\cap[\alpha,\beta]\neq\emptyset$. Choosing $\lambda\in I_{\mathrm{ad}}$ yields $f(x^+)-f(x)\leq-(1/\delta)|D|f(x)^{\frac{2}{p-1}}$, and in particular $f(x^+)<f(x)$ as well as $x^+\in V$.
\end{proof}
The proof of Theorem \ref{th:suff descent - suff} now follows directly from  Proposition \ref{pr: SD method}. We emphasize that the constants $\overline{D}$ and $\underline{\delta}$ are fixed, once the initial point $x^0$ is chosen. In particular, $I_{\mathrm{ad}}^k$ is nonempty for each iterate $k$.
\subsection{Almost steepest descent directions}\label{s:inexact}
As already noted, the steepest descent direction $\gamma_{\min}$ characterized by \eqref{eq: steepest descent} is the solution to an optimization problem which has the same form as the original optimization problem. We develop in this section an {\em almost} steepest descent method that still satisfies the sufficient decrease condition \eqref{eq: BD assm suff descent} at least up to the $K$-th iterate, with error estimates for $f(x^K)$ as well as $x^K$. The accuracy can be measured either relatively or absolutely.  For a relative estimate, the descent step $y^k\in\mathcal{D}_k$ is chosen such that
\begin{equation}\label{eq: rel error}
    Df(x^k)[\gamma_{[x^k,y^k]}]=-r_k|D|f(x^k) \qquad\text{for all}\,\,k\in\bbN
\end{equation}
with rate $r_k\in(0,1]$. An absolute estimate takes the form
\begin{equation}\label{eq: abs error}
    \varepsilon_k:=Df(x^k)[\gamma_{[x^k,y^k]}]+|D|f(x^k)=(1-r_k)|D|f(x^k) \qquad\text{for all}\,\,k\in\bbN.
\end{equation}
\begin{lemma}\label{lm: ID method 1}
    Let $(G,d)$ be boundedly compact, $p$-uniformly convex with $p\in(1,2]$. Assume that
    \begin{enumerate}[(i)]
        \item $f:G\to\mathbb{R}$ is level-bounded and has local uniform Lipschitz gradient along geodesics with constant $L>0$ on the convex subset $U$;
        \item $\mathrm{lev}_{\leq f(x^0)}(f)\subset \mathrm{gi}\,(U)$ with distance at least $\rho>0$ to the boundary of $U$: i.e. $\bigcup_{x\in \mathrm{lev}_{\leq f(x^0)}(f)} B_\rho(x)\subset \mathrm{gi}\,(U)$;
        \item for $x\in\mathrm{lev}_{\leq f(x^0)}(f)$ the step $y\in\mathcal{D}(x)$ satisfies
        \begin{equation}\label{eq: ID method y^k}
            Df(x)[\gamma_{[x,y]}]=-r|D|f(x)\quad\text{and}\quad d(x,y)=\rho
        \end{equation}
        with rate $r\in(0,1]$.
    \end{enumerate}
    Then there exists a uniform constant $\underline{\delta}_r>0$ such that for all $\delta\geq\underline{\delta}_r$, the interval $I_{\mathrm{ad}}$ is nonempty, and for any $\lambda\in I_{\mathrm{ad}}$
    $$
    f(x^+)-f(x)\leq-\frac{1}{\delta}|D|f(x)^{\frac{2}{p-1}} \qquad \text{with}\,\,x^+:=\gamma_{[x,y]}(\lambda).
    $$
\end{lemma}
\begin{proof}
    There is a uniform upper bound
    $$
    |D|f(x)\leq\overline{D}<+\infty \qquad\text{for all}\,\,x\in V:=\mathrm{lev}_{\leq f(x^0)}(f).
    $$
    Now we fix a $\underline{\delta}_r>0$ sufficiently large such that
    \begin{equation}\label{eq: ID method delta 1}
        rs\geq\left(\frac{2+L}{2\sqrt{\underline{\delta}_r}}\right)s^{\frac{1}{p-1}} \qquad\text{for all}\,\,0\leq s\leq\overline{D}
    \end{equation}
    and
    \begin{equation}\label{eq: ID method delta 2}
        \left(\frac{2+L}{2r\rho\underline{\delta}_r}\right)\overline{D}^{\frac{2}{p-1}-1}\leq 1.
    \end{equation}
    Then for all $\delta\geq\underline{\delta}_r$ we have
    $$
    Df(x)[\gamma_{[x,y]}]=-r|D|f(x)\leq-\left(\frac{2+L}{2\sqrt{\delta}}\right)|D|f(x)^{\frac{1}{p-1}}
    $$
    as well as
    $$
    \alpha=\left(\frac{2+L}{2\delta}\right)\frac{|D|f(x)^{\frac{2}{p-1}-1}}{rd(x,y)}=\left(\frac{2+L}{2r\rho\delta}\right)|D|f(x)^{\frac{2}{p-1}-1}\leq 1.
    $$
    Thus $I_{\mathrm{ad}}\neq\emptyset$ and choosing $\lambda\in I_{\mathrm{ad}}$ yields
    $$
    f(x^+)-f(x)\leq-\frac{1}{\delta}|D|f(x)^{\frac{2}{p-1}}.
    $$
    In particular,  $f(x^+)<f(x)$ as well as $x^+\in V$.
\end{proof}
In order to apply Lemma \ref{lm: ID method 1} in each step $k\in\bbN$, we assume that the sequence $(r_k)_{k\in\bbN}$ is bounded away from zero.
\begin{theorem}[almost steepest descent method: relative errors]\label{th: ID method 1}
    Let $(G,d)$ be boundedly compact, $p$-uniformly convex with $p\in(1,2]$. Assume that
    \begin{enumerate}[(i)]
        \item $f:G\to\mathbb{R}$ is level-bounded and has local uniform Lipschitz gradient along geodesics with constant $L>0$ on the convex subset $U$;
        \item $\mathrm{lev}_{\leq f(x^0)}(f)\subset \mathrm{gi}\,(U)$ with distance at least $\rho>0$ to the boundary of $U$: i.e. $\bigcup_{x\in \mathrm{lev}_{\leq f(x^0)}(f)} B_\rho(x)\subset \mathrm{gi}\,(U)$;
        \item in each step $k$, the direction $y^k\in\mathcal{D}_k$ satisfies
        $$
        Df(x^k)[\gamma_{[x^k,y^k]}]=-r_k|D|f(x^k) \quad\text{and}\quad d(x^k,y^k)=\rho;
        $$
        \item the sequence $(r_k)_{k\in\bbN}$ is bounded away from zero, i.e. there exist $r>0$ such that $0<r<r_k\leq 1$ for all $k\in\bbN$.
    \end{enumerate}
    Then there exists a uniform constant $\underline{\delta}_r>0$ such that, for all $\delta\geq\underline{\delta}_r$ and all $k\in\bbN$, the set $I_{\mathrm{ad}}^k\neq\emptyset$, and when $\lambda_k\in I_{\mathrm{ad}}^k$ for all $k$, the sufficient descent condition holds:
    $$
    f(x^{k+1})-f(x^k)\leq-\frac{1}{\delta}|D|f(x^k)^{\frac{2}{p-1}} \qquad\text{for all}\,\,k.
    $$
\end{theorem}
\begin{proof}
    For the upper bound
    $$
    |D|f(x)\leq\overline{D}<+\infty \qquad\text{for all}\,\,x\in V:=\mathrm{lev}_{\leq f(x^0)}(f)
    $$
    we fix a $\underline{\delta}_r>0$ sufficiently large such that
    $$
    rs\geq\left(\frac{2+L}{2\sqrt{\underline{\delta}_r}}\right)s^{\frac{1}{p-1}} \qquad\text{for all}\,\,0\leq s\leq\overline{D}
    $$
    and
    $$
    \left(\frac{2+L}{2r\rho\underline{\delta}_r}\right)\overline{D}^{\frac{2}{p-1}-1}\leq 1.
    $$
    Then for all $k$ we have
    $$
    Df(x^k)[\gamma_{[x^k,y^k]}]=-r_k|D|f(x^k)<-r|D|f(x^k)\leq-\left(\frac{2+L}{2\sqrt{\delta}}\right)|D|f(x^k)^{\frac{1}{p-1}}
    $$
    as well as
    $$
    \alpha_k=\left(\frac{2+L}{2\delta}\right)\frac{|D|f(x^k)^{\frac{2}{p-1}-1}}{r_kd(x^k,y^k)}<\left(\frac{2+L}{2r\rho\delta}\right)|D|f(x^k)^{\frac{2}{p-1}-1}\leq 1.
    $$
    Thus $I^k_{\mathrm{ad}}\neq\emptyset$ and the rest follows from Lemma \ref{lm: ID method 1}.
\end{proof}
Next we state a more general accuracy assumption. We first point out that, under the assumptions of Theorem \ref{th: ID method 1},
$$
\varepsilon_k=(1-r_k)|D|f(x^k)\leq (1-r)\overline{D} \qquad\text{for all}\,\,k\in\bbN.
$$
Without the assumption on the sequence $(r_k)_{k\in\bbN}$ we directly assume that the sequence of absolute errors $(\varepsilon_k)_{k\in\bbN}$ is bounded above, i.e. $\varepsilon_k\leq\varepsilon$ for all $k\in\bbN$ with some small $\varepsilon>0$.
\begin{theorem}[almost steepest descent method: absolute error]\label{th: ID method 2}
    Let $(G,d)$ be boundedly compact, $p$-uniformly convex with $p\in(1,2]$. Assume that
    \begin{enumerate}[(i)]
        \item $f:G\to\mathbb{R}$ is level-bounded and has local uniform Lipschitz gradient along geodesics with constant $L>0$ on the convex subset $U$;
        \item $\mathrm{lev}_{\leq f(x^0)}(f)\subset \mathrm{gi}\,(U)$ with distance at least $\rho>0$ to the boundary of $U$: i.e. $\bigcup_{x\in \mathrm{lev}_{\leq f(x^0)}(f)} B_\rho(x)\subset \mathrm{gi}\,(U)$;
        \item in each step $k$ the step $y^k\in\mathcal{D}_k$ satisfies
        \begin{equation}\label{eq: ID method y^k 2}
            Df(x^k)[\gamma_{[x^k,y^k]}]=-r_k|D|f(x^k)\quad\text{and}\quad d(x^k,y^k)=\rho
        \end{equation}
        with rate $r_k\in(0,1]$;
        \item the sequence $(\varepsilon_k)_{k\in\bbN}$ given by \eqref{eq: abs error} is bounded from above, i.e. there exist $\varepsilon>0$ such that $0\leq\varepsilon_k\leq\varepsilon$ for all $k\in\bbN$.
    \end{enumerate}
    Then there exists a uniform constant $\underline{\delta}_\varepsilon>0$ such that for all $\delta\geq\underline{\delta}_\varepsilon$, the interval $I_{\mathrm{ad}}^k$ is nonempty for all $0\leq k\leq K-1$ for some $K\in\bbN$, and for any $\lambda_k\in I_{\mathrm{ad}}^k$ the sufficient decrease condition \eqref{eq: BD assm suff descent} holds with constant $\delta$ for all $0\leq k\leq K-1$. Furthermore, if $\mathrm{argmin}_U(f)\neq\emptyset$, the PŁ property with constant $C_U>0$ yields the error estimate
    $$
    f(x^K)-f_{\min}\leq\Bigl(f(x^0)-f_{\min}\Bigr)\left(1-\frac{C_U}{\delta}\right)^K.
    $$
\end{theorem}
\begin{proof}
    Again, there is a uniform upper bound
    $$
    |D|f(x)\leq\overline{D}<+\infty \qquad\text{for all}\,\,x\in V:=\mathrm{lev}_{\leq f(x^0)}(f).
    $$
    Now we fix a $\underline{\delta}_\varepsilon>0$ sufficiently large such that
    \begin{equation}\label{eq: ID method delta 3}
        s-\varepsilon\geq\left(\frac{2+L}{2\sqrt{\underline{\delta}_\varepsilon}}\right)s^{\frac{1}{p-1}} \qquad\text{for all}\,\,\varepsilon<s^*\leq s\leq\overline{D}
    \end{equation}
    as well as
    \begin{equation}\label{eq: ID method delta 4}
        \left(\frac{2+L}{2\rho\underline{\delta}_\varepsilon}\right)\frac{s^{\frac{2}{p-1}}}{s-\varepsilon}\leq 1 \qquad\text{for all}\,\,\varepsilon<s^*\leq s\leq\overline{D}
    \end{equation}
    for some $s^*>0$.
    Then for $\underline{\delta}_\varepsilon$ satisfying \eqref{eq: ID method delta 3} and \eqref{eq: ID method delta 4} we have
    $$
    Df(x^0)[\gamma_{[x^0,y^0]}]=-|D|f(x^0)+\varepsilon_k\leq-|D|f(x^0)+\varepsilon\leq-\left(\frac{2+L}{2\sqrt{\underline{\delta}_\varepsilon}}\right)|D|f(x^0)^{\frac{1}{p-1}}
    $$
    as well as
    $$
    \alpha_0=\left(\frac{2+L}{2\rho\underline{\delta}_\varepsilon}\right)\frac{|D|f(x^0)^{\frac{2}{p-1}}}{|D|f(x^0)-\varepsilon_k}\leq\left(\frac{2+L}{2\rho\underline{\delta}_\varepsilon}\right)\frac{|D|f(x^0)^{\frac{2}{p-1}}}{|D|f(x^0)-\varepsilon}\leq 1
    $$
    provided that $\varepsilon<s^*\leq|D|f(x^0)\leq\overline{D}$. Thus $I_{\mathrm{ad}}^0\neq\emptyset$ and choosing $\lambda_0\in I_{\mathrm{ad}}^0$ yields
    $$
    f(x^1)-f(x^0)\leq-\frac{1}{\underline{\delta}_\varepsilon}|D|f(x^0)^{\frac{2}{p-1}}.
    $$
    In particular,  $f(x^1)<f(x^0)$ as well as $x^1\in V$. Assume that we can provide these steps until step $K-1$, i.e.
    $$
    \varepsilon<s^*\leq|D|f(x^k)\leq\overline{D} \qquad\text{for all}\,\,0\leq k\leq K-1.
    $$
    After choosing $\lambda_{K-1}\in I_{\mathrm{ad}}^{K-1}$, we have
    $$
    f(x^{k+1})-f(x^k)\leq-\frac{1}{\delta}|D|f(x^k)^{\frac{2}{p-1}} \qquad \text{for all}\,\,0\leq k\leq K-1.
    $$
    Furthermore, the proof of Theorem \ref{th: cvg func value} yields the error estimate
    $$
    f(x^K)-f_{\min}\leq\Bigl(f(x^0)-f_{\min}\Bigr)\left(1-\frac{C_U}{\delta}\right)^K.
    $$
\end{proof}
\begin{remark}
    The constant $s^*$ and the step number $K$ depend on the upper bound $\varepsilon$, yet there is no known explicit expression. But $|s^*-\varepsilon|$ can be arbitrarily small, e.g., $s^*$ solves the equation
    $$
    s^*-\varepsilon=\left(\frac{2+L}{2\sqrt{\underline{\delta}_\varepsilon}}\right)(s^*)^{\frac{1}{p-1}}
    $$
    It follows that $|D|f(x^K)\sim\varepsilon$. Additionally, if the function $f$ is strongly convex with constant $\mu>0$,
    by Proposition \ref{pr: dir dervt of cvx fct} and the proof of Corollary \ref{cl: cvg ite for strongly cvx func}
    $$
    d(x^K,\bar{x})\leq\sqrt[p]{\frac{2}{\mu}\Bigl(f(x^0)-f_{\min}\Bigr)}\left(\sqrt[p]{1-\frac{C_U}{\delta}}\right)^{K}.
    $$
\end{remark}
\section{Examples}\label{s:examples}
Before moving to concrete examples, it is worth noting that computing a steepest descent direction could be
just as hard as solving the optimization problem to which this is to be applied.
Consider the problem of finding the minimum of a sum of functions,
\begin{equation}\label{eq:min sum}
 \min_{x\in G}\sum_{j=1}^m f_j(x),
\end{equation}
For $f(x)\equiv \sum_{j=1}^m f_j(x)$, and $f_j$ directionally
differentiable on a neighborhood of the point $x$,
by the sum rule, Proposition \ref{pr: sum rules}, adapted to the normalized directional derivative, 
\begin{equation}\begin{aligned}
                    Df(x)[\gamma] &= D\left(\sum_{j=1}^m f_j(x)\right)[\gamma]=\sum_{j=1}^m Df_j(x)[\gamma],\quad \gamma\in G_x^\circ.
                \end{aligned}
\end{equation}
Then the problem of computing the steepest descent direction, \eqref{eq: steepest descent}, is
\begin{equation}\label{eq:sd sum}
 \min_{\gamma\in G_x^\circ}\sum_{j=1}^m Df_j(x)[\gamma].
\end{equation}
Because the space $(G_x, d_x)$ defined by \eqref{eq:(G_x,d_x)} is of the same type as $(G, d)$,
the problem of computing the steepest descent is the same as the original problem \eqref{eq:min sum},
with the exception that the domain over which the optimization is performed is {\em not closed} because trivial geodesics are excluded.
Proposition \ref{pr: ex of SD direction}, nevertheless, provides guarantees for existence of solutions to \eqref{eq:sd sum}.

\subsection{Fr\'echet means}
All of the assumptions of the theory can be verified explicitly for the concrete case of Fr\'echet means.  
This example also demonstrates the point above:  to compute the steepest descent direction of the Fr\'echet 
function at a point $x$, one has to compute the Fr\'echet mean on a sphere centered at $x$.  

Indeed, in
a $p$ uniformly convex space with constants $p\in (1,\infty)$ and $c>0$, each
function in \eqref{eq:min sum} is given by $f_j(x)=d(x, x_j)^p$, and
\begin{equation}\label{eq:sd Frechet}
\begin{aligned}
        &Df_j(x)[\gamma_{[x,y]}]\\
        &\quad =\frac{df_j(x)[\gamma_{[x,y]}]}{d(x, y)}\\
        &\quad  = \lim_{t\searrow 0}\frac{d(\gamma_{[x,y]}(t), x_j)^p- d(x, x_j)^p}{d(x, y) t}\\
        &\quad  = \lim_{t\searrow 0}\frac{d((1-t)x\oplus t y, x_j)^p - d(x, x_j)^p}{d(x, y) t}\\
        &\quad  \leq \lim_{t\searrow 0}\frac{(1-t)d(x, x_j)^p + td(y, x_j)^p - \frac{c}{2}(1-t)td(x,y)^p - d(x, x_j)^p}{d(x, y) t}\\
        &\quad   = \frac{1}{d(x, y) }\lim_{t\searrow 0} \left(d(y, x_j)^p - d(x, x_j)^p - \frac{c}{2}(1-t)d(x,y)^p\right)\\
        &\quad  = \frac{1}{d(x, y)  } \left(d(y, x_j)^p - d(x, x_j)^p - \frac{c}{2}d(x,y)^p\right),
\end{aligned}
\end{equation}
where the inequality follows from the definition of a $p$-uniformly convex space \ref{eq: (p,c)-space}.  The directional derivative
is therefore given directly in terms of the endpoint $y$ of directions $\gamma_{[x,y]}$.   Restricting the search for the steepest
descent direction over a sphere of radius $\rho$ around the point $x$, $\mathbb{S}_{x}(\rho)$, for $\rho$ small enough as in
Proposition \ref{pr: ex of SD direction} we have
$Df_j(x)[\gamma_{[x,y]}]\leq \frac{1}{\rho } \left(d(y, x_j)^p - d(x, x_j)^p - \frac{c}{2}\rho^p\right)$ for
all $y\in \mathbb{S}_{x}(\rho)$.  The minimum of the right-hand side on 
$\mathbb{S}_{x}(\rho)$ is simply the intersection of the geodesic $\gamma_{[x,x_j]}$ 
with the sphere.  The steepest descent direction for the function $f_j(x)=d(x, x_j)^p$
is also the geodesic $\gamma_{[x,x_j]}$.  So the steepest descent direction on the 
sphere coincides. 

For a sum of functions, consider the 
solution $\overline{y}$ to the constrained optimization problem
\begin{equation}\label{eq:sd sum 2}
 \min_{y\in \mathbb{S}_x(\rho)}\sum_{j=1}^m  d(y, x_j)^p.
\end{equation}
This is a Fr\'echet mean problem restricted to the sphere of radius $\rho$ centered at $x$.
This provides and upper bound on the value of the normalized directional derivative 
in the direction $\gamma_{[x,\overline{y}]}$, i.e.
\begin{equation*}
\begin{aligned}
\sum_{j=1}^m Df_j[\gamma_{[x,\overline{y}]}]\leq& 
\tfrac{1}{\rho}\sum_{j=1}^m \left(d(\overline{y}, x_j)- d(x, x_j)^p - \frac{c}{2}\rho^p\right)\\
\leq &
\tfrac{1}{\rho}\sum_{j=1}^m \left(d(y, x_j)- d(x, x_j)^p - \frac{c}{2}\rho^p\right) \quad\forall y\in \mathbb{S}_x(\rho).
\end{aligned}
\end{equation*} 
This demonstrates that estimating the steepest descent direction for the Fr\'echet 
function involves solving a 
problem in exactly the same form as the original problem of 
minimizing the Fr\'echet function.  

\subsection{Smooth functions in nonsmooth spaces}\label{sec: limitation of regularity}
While in classic linear spaces the class of functions with Lipschitz gradient
is rich enough that a huge number of practical applications can be modeled with such
functions, in a metric space the local topology may impose further restrictions.
We present a concrete construction of a space where the condition
of $C^1$ along geodesics uniquely determines the value of the directional
derivatives at a point;  in other words, only very special functions in the given
space are $C^1$ along geodesics.

Consider the space $\{a, b, c\} \times [0. 1]$ and define on it the semimetric
\begin{equation}\label{eq:definition-metric-mercedes}
  d((i, x), (j, y)) =
  \left\{\begin{array}{ll}
    |x - y|   & \text{ if } i = j                     \\
    1 - x + y & \text{ if } i = a \text{ and } j = b  \\
    1 - x + y & \text{ if } i = a \text{ and } j = c  \\
    x + 1 - y & \text{ if } i = b \text{ and } j = a  \\
    x + y     & \text{ if } i = b \text{ and } j = c  \\
    x + 1 - y & \text{ if } i = c \text{ and } j = a  \\
    x + y     & \text{ if } i = c \text{ and } j = b. \\
  \end{array}\right.
\end{equation}
The metric space $G$ is then defined as $\{a, b, c\} \times [0, 1] / d$. This has the
effect of identifying $(a, 1) = (b, 0) = (c, 0)$. Geometrically, this space is
represented in Figure~\ref{fig:mercedes}

\begin{figure}
    \centering
    \begin{tikzpicture}
        \node {$(a, 0)$}
        child { node {$(a, 1) = (b, 0) = (c, 0)$}
        child { node {$(b, 1)$}}
        child { node {$(c, 1)$}}
        };
    \end{tikzpicture}
  \caption{Representation of the space $\{a, b, c\} \times [0, 1] / d$.}\label{fig:mercedes}
\end{figure}
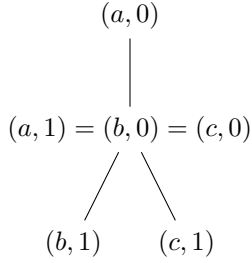

\begin{proposition}
  Let $(G, d)$ be the space defined by~\eqref{eq:definition-metric-mercedes}.
  The only functions $f:G \to \mathbb{R}$ that are in $C^1(G)$  are 
  functions with a critical point at $(a, 1) = (b, 0) = (c, 0)$. 
\end{proposition}
\begin{proof}
The claim is proved by showing that the directional derivative in all possible 
directions is zero at the point $(a, 1)$.   

Consider first the geodesics
  $\gamma_{[(a, 0), (a, 1)]}$, $\gamma_{[(b, 1), (b, 0)]}$, $\gamma_{[(c, 1), (c, 0)]}$ and use
  the property of being $C^1$ along geodesics to conclude that there are $C^1$ functions
  $f_\alpha, f_\beta, f_\gamma:[0, 1] \to \bbR$ such that
  \begin{equation*}
    f(i, x) =
    \left\{
      \begin{array}{ll}
        f_\alpha(x) & \text{ if } i = a  \\
        f_\beta(x) & \text{ if } i = b  \\
        f_\gamma(x) & \text{ if } i = c. \\
      \end{array}
    \right.
  \end{equation*}

  The next step is to compute the directional derivatives, expressing
  them using the derivatives of $f_\alpha,f_\beta$ and $f_\gamma$.
This computation yields
  \begin{equation*}
    df(i, x)[\gamma_{[(i, x), (j, y)]}] =
    \left\{
      \begin{array}{ll}
        f_\alpha'(x) (y - x)
        & \!\!\text{if } i = a \text{ and } j = a \text{ and } x \in [0, 1] \\
        f_\alpha'(x) (1 - x + y)
        & \!\!\text{if } i = a \text{ and } j = b \text{ and } x \in [0, 1) \\
        f_\alpha'(x) (1 - x + y)
        & \!\!\text{if } i = a \text{ and } j = c \text{ and } x \in [0, 1) \\
        - f_\beta'(x) (x + 1 - y)
        & \!\!\text{if } i = b \text{ and } j = a \text{ and } x \in (0, 1]  \\
        f_\beta'(x) (y - x)
        & \!\!\text{if } i = b \text{ and } j = b \text{ and } x \in [0, 1] \\
        - f_\beta'(x) (x + y)
        & \!\!\text{if } i = b \text{ and } j = c \text{ and } x \in (0, 1] \\
        - f_\gamma'(x) (x + 1 - y)
        & \!\!\text{if } i = c \text{ and } j = a \text{ and } x \in (0, 1] \\
        - f_\gamma'(x) (x + y)
        & \!\!\text{if } i = c \text{ and } j = b \text{ and } x \in (0, 1] \\
        f_\gamma'(x) (y - x)
        & \!\!\text{if } i = c \text{ and } j = c \text{ and } x \in [0, 1] \\
      \end{array}
    \right. .
  \end{equation*}

  The point $(a, 1) = (b, 0) = (c, 0)$ is of particular interest for the
  topology of its neighborhoods. There are three independent geodesics starting
  from this point, and evaluating the corresponding three directional derivatives yields
  \begin{align*}
    df(a, 1)[\gamma_{[(a, 1), (a, 0)]}] &= f_\alpha'(1) \\
    df(b, 0)[\gamma_{[(b, 0), (b, 1)]}] &= f_\beta'(0) \\
    df(c, 0)[\gamma_{[(c, 0), (c, 1)]}] &= f_\gamma'(0).
  \end{align*}

  In order to obtain the constraints on $f$, we evaluate the directional
  derivatives at points in a neighborhood of the critical point, along geodesics
  passing through this point
  \begin{align*}
    df(a, 1-\varepsilon)[\gamma_{[(a, 1-\varepsilon), (b, 1)]}] &= f_\alpha'(1-\varepsilon) (1 + \varepsilon) \\
    df(a, 1-\varepsilon)[\gamma_{[(a, 1-\varepsilon), (c, 1)]}] &= f_\alpha'(1-\varepsilon) (1 + \varepsilon) \\
    df(b, \varepsilon)[\gamma_{[(b, \varepsilon), (a, 0)]}]     &= - f_\beta'(\varepsilon) (1 + \varepsilon) \\
    df(b, \varepsilon)[\gamma_{[(b, \varepsilon), (c, 1)]}]     &= - f_\beta'(\varepsilon) (1 + \varepsilon) \\
    df(c, \varepsilon)[\gamma_{[(c, \varepsilon), (a, 0)]}]     &= - f_\gamma'(\varepsilon) (1 + \varepsilon) \\
    df(c, \varepsilon)[\gamma_{[(c, \varepsilon), (b, 1)]}]     &= - f_\gamma'(\varepsilon) (1 + \varepsilon).
  \end{align*}

  Since $f$ is $C^1$ along geodesics, we can take the limits as $\varepsilon \downarrow 0$ and
  we obtain the system of equations
  \begin{align*}
    f_\beta'(0) &=   f_\alpha'(1) \\
    f_\gamma'(0) &=   f_\alpha'(1) \\
    f_\alpha'(1) &= - f_\beta'(0) \\
    f_\gamma'(0) &= - f_\beta'(0) \\
    f_\alpha'(1) &= - f_\gamma'(0) \\
    f_\beta'(0) &= - f_\gamma'(0).
  \end{align*}

  This shows that $f_\alpha'(1) = f_\beta'(0) = f_\gamma'(0) = 0$ and that
  \begin{equation*}
    \forall j \in \{a, b, c\}, y \in [0, 1]\quad df(a, 1)[\gamma_{[(a, 1), (j, y)]}] = 0.
  \end{equation*}
  In other words, $f$ has a critical point at $(a,1)$ as claimed. 
\end{proof}

\begin{example}[Fr\'echet functions]
  In particular, consider solving a barycenter problem
  by minimizing $f:G \to \bbR$ defined by
\begin{equation*}
  f(i, x) = \alpha d((i, x), (\alpha, 0))^2 + \beta d((i, x), (b, 1))^2 + \gamma d((i, x), (b, 1))^2.
\end{equation*}
It is an interesting exercise to show that the only function that is $C^1$ along geodesics
has $\alpha = \beta = \gamma = 0$, that is, the zero function.
\end{example}
The counterexamples above should come as no surprise:  nonsmoothness of the
{\em space} prohibits nontrivial smooth variation of functions.  As long
the functions can be pieced together by local restrictions to finite dimensional
complete, locally compact inner metric spaces with bounded curvature (i.e. manifolds \cite[Theorem A]{Plaut92}),
then there should be sufficient modelling power for functions with local uniform Lipschitz gradient along geodesics. 
This indicates that {\em pointwise} smoothness at points where curvature is $-\infty$ could be a way forward 
for the analysis of algorithms that traverse such points.  
A systematic study
of such classes of functions is open.


\end{document}